\documentclass[11pt]{amsart}

\usepackage[a4paper,hmargin=3.5cm,vmargin=4cm]{geometry}
\usepackage{amsfonts,amssymb,amscd,amstext}
\usepackage{graphicx}
\usepackage[dvips]{epsfig}
\usepackage{quoting}
\usepackage{hyperref}

\usepackage{fancyhdr}
\input xy
\xyoption{all}
\usepackage{tikz}
\usetikzlibrary{matrix}

\usepackage[color=blue!20]{todonotes}

\usepackage{enumitem}
\usepackage{titlesec}
\usepackage{mathrsfs}

\titleformat{\section}
{\filcenter\bfseries} {\thesection{.}}{0.2cm}{}
\titleformat{\subsection}[runin]
{\bfseries} {\thesubsection{.}}{0.15cm}{}[.]
\titleformat{\subsubsection}[runin]
{\em}{\thesubsubsection{.}}{0.15cm}{}[.]

\usepackage[up,bf]{caption}

\newtheorem{theorem}{Theorem}[section]

\newtheorem{lemma}[theorem]{Lemma}
\newtheorem{claim}[theorem]{Claim}
\newtheorem{corollary}[theorem]{Corollary}

\theoremstyle{definition}
\newtheorem{definition}[theorem]{Definition}
\newtheorem{remark}[theorem]{Remark}

\numberwithin{equation}{section}
\numberwithin{figure}{section}

\newcommand\Ccal{\mathcal{C}}

\newcommand\Hcal{\mathcal{H}}

\newcommand\Rcal{\mathcal{R}}

\newcommand\Tcal{\mathcal{T}}

\newcommand\Ascr{\mathscr{A}}

\newcommand\Cscr{\mathscr{C}}

\newcommand\Oscr{\mathscr{O}}

\newcommand\Pscr{\mathscr{P}}

\newcommand\C{\mathbb{C}}
\newcommand\D{\overline{\mathbb D}}
\newcommand\CP{\mathbb{CP}}
\renewcommand\D{\mathbb D}

\newcommand\N{\mathbb{N}}

\newcommand\R{\mathbb{R}}

\renewcommand\S{\mathbb{S}}

\newcommand\Z{\mathbb{Z}}

\renewcommand\c{\mathbb{C}}

\renewcommand\d{\mathbb D}

\newcommand\h{\mathbb{H}}
\newcommand\n{\mathbb{N}}
\renewcommand\r{\mathbb{R}}
\newcommand\s{\mathbb{S}}

\newcommand\z{\mathbb{Z}}

\newcommand\igot{\mathfrak{i}}

\renewcommand\igot{\mathfrak{i}}

\newcommand\wt{\widetilde}
\newcommand\wh{\widehat}

\newcommand\dist{\mathrm{dist}}

\newcommand\length{\mathrm{length}}

\newcommand\CMC{\mathrm{CMC\text{-}1}}

\newcommand{\ord}{{\rm ord}}

\def\dist{\mathrm{dist}}

\def\length{\mathrm{length}}

\usepackage{color}

\usepackage{array}

\begin{document}


\fancyhead[LO]{Complete CMC-1 surfaces in hyperbolic space}
\fancyhead[RE]{A.\ Alarc\'on, I.\ Castro-Infantes, and J.\ Hidalgo}
\fancyhead[RO,LE]{\thepage}

\thispagestyle{empty}



\begin{center}
{\bf\Large Complete CMC-1 surfaces in hyperbolic space\\ \smallskip with arbitrary complex structure
} 

\bigskip

%
%
{\bf Antonio Alarc\'on,
Ildefonso Castro-Infantes, and
Jorge Hidalgo}
\end{center}

\smallskip

\begin{center}
{\em Dedicated to Francisco J.\ L\'opez for his sixtieth birthday}
\end{center}

\medskip

%
%

\bigskip

\begin{quoting}[leftmargin={7mm}]
{\small
\noindent {\bf Abstract}\hspace*{0.1cm}
We prove that every open Riemann surface $M$ is the 
complex 
structure of a complete surface of constant mean curvature $1$ ($\CMC$) in the 3-dimensional 
hyperbolic space $\h^3$. We go further and establish a 
jet interpolation theorem for complete conformal $\CMC$ immersions $M\to \h^3$. As a consequence, we show the existence of complete densely immersed $\CMC$ surfaces in $\h^3$ with arbitrary complex structure.
We obtain these results as application of a  
uniform approximation theorem with 
jet interpolation for holomorphic null curves in $\c^2\times\c^*$ which is also established in this paper.

\noindent{\bf Keywords}\hspace*{0.1cm} 
Constant mean curvature surface, Bryant surface, regular end, Riemann surface, holomorphic null curve,  minimal surface.


\noindent{\bf Mathematics Subject Classification (2020)}\hspace*{0.1cm} 
53A10, 
53C42, 
32H02, 
32E30. 
}
\end{quoting}



\section{Introduction}
\label{sec:intro}

\noindent  In this paper we shall establish that every open Riemann surface $M$ is the underlying complex structure of a complete surface of constant mean curvature $1$ ($\CMC$) in the 3-dimensional hyperbolic space $\h^3$ (see Corollary \ref{co:intro-dense}).   We prove more: our first main result is the following 
jet interpolation theorem for complete conformal $\CMC$ immersions $M\to\h^3$ with some additional control on the asymptotic behavior. 
%
%
\begin{theorem}\label{th:intro-Bryant}
Let $M$ be an open Riemann surface and $\Lambda$ and $E$ be a pair of disjoint closed discrete subsets of $M$. Given a conformal $\CMC$ immersion $\varphi: U\to\h^3$ on a neighborhood $U\subset M$ of $\Lambda$ and maps $k:\Lambda\to\n=\{1,2,3,\ldots\}$ and $m: E \to\n$ such that $m$ only assumes odd values, there is a complete conformal $\CMC$ immersion $\psi: M\setminus E\to\h^3$ satisfying the following conditions:
\begin{enumerate}[label= \rm (\roman*)]
\item \label{prop1} $\psi$ and $\varphi$ have a contact of order $k(p)$ at every point $p\in\Lambda$.

\item \label{prop2} For each point $p\in E$ the end of $\psi$ corresponding to $p$ is of finite total curvature, regular, and of multiplicity $m(p)$.

\item \label{prop3} For each point $p\in m^{-1}(1)\subset E$ the end of $\psi$ corresponding to $p$ is smooth.
\end{enumerate}
Furthermore, $\psi: M\setminus E\to\h^3$ can be chosen to be an almost proper map.
\end{theorem}

Theorem \ref{th:intro-Bryant} is proved in Section \ref{sec:bryant}. 
We do not know whether the statement remains to hold if the map $m$ is allowed to also take even values or if the immersion $\psi$ is asked to have embedded non-smooth ends at some points in $m^{-1}(1)$. We expect it does, although our construction method does not seem to work in these cases (see Lemmas \ref{lemma::m} and \ref{lemma::smooth}). Recall that a continuous map $f:X\to Y$ between topological spaces is said to be {\em almost proper} if  for every compact set $K\subset Y$ the connected components of $f^{-1}(K)$ are all compact. Every almost properly immersed submanifold of an open complete Riemannian manifold is complete. In particular, almost properly immersed surfaces in $\h^3$ are complete. 
 
Surfaces of constant mean curvature $1$ in $\h^3$ are often called {\em Bryant surfaces} after Robert Bryant for his seminal paper \cite{Bryant1987Asterisque}; we shall adopt this terminology. Bryant showed that these surfaces admit a conformal representation in terms of holomorphic data on an open Riemann surface (see Section \ref{sec:SL2C}) that is analogous to the Weierstrass representation formula for minimal surfaces in the Euclidean space $\r^3$. This allows to study Bryant surfaces by using complex analytic methods. Moreover, similar to minimal surfaces in $\r^3$ which are characterized by the holomorphicity of their Gauss map, Bryant surfaces are characterized by their {\em hyperbolic Gauss map} being holomorphic \cite{Bryant1987Asterisque}. Complete Bryant ends of finite total curvature are conformally equivalent to a punctured disc, but, unlike the case of minimal surfaces, the hyperbolic Gauss map of such an end need not extend to the puncture. Complete Bryant ends at which the hyperbolic Gauss map extends meromorphically are called {\em regular} \cite[Def.\ 1.4]{UmeharaYamada1993AM}. Since the hyperbolic Gauss map is fundamental to the geometry of a Bryant surface, the study of regular ends, starting with the landmark paper by Umehara and Yamada \cite{UmeharaYamada1993AM}, is of great importance in the theory. The {\em multiplicity} $m \in \N$ of a Bryant regular end (see Section \ref{sec:bryant} or \cite[\textsection5]{UmeharaYamada1993AM} for the precise notion) measures how far is the end to be embedded; such an end is embedded if and only if it is of multiplicity $1$ \cite[Theorem 5.2]{UmeharaYamada1993AM}. Bryant regular ends are tangent to the ideal boundary $\partial_{\infty}\h^3$ of $\h^3$, while an immersed Bryant annular end is said to be {\em smooth} if it is conformally equivalent to a punctured disc and the immersion extends smoothly to the puncture through $\partial_{\infty}\h^3$ (see Bohle and Peters \cite[p.\ 590]{BohlePeters2009}). 
Smooth Bryant ends are properly embedded (hence of multiplicity one), regular, and of finite total curvature. More generally, every properly embedded Bryant annular end is regular and of finite total curvature by the fundamental result of Collin, Hauswirth, and Rosenberg \cite[Theorem 10]{CollinHauswirthRosenberg2001AM}. 

The set $E$ in Theorem \ref{th:intro-Bryant} is allowed to be empty; conditions \ref{prop2} and \ref{prop3} are vacuous in this particular case. We thus obtain the following immediate corollary which proves that every open Riemann surface is the complex structure of a complete Bryant surface (see Corollary \ref{cor::intbryant} for a more precise statement).
%
%
\begin{corollary}\label{co:intro-dense}
Let $M$ be an open Riemann surface. For any divergent sequence $a_1,a_2,\ldots$ in $M$ without repetition and any sequence $c_1,c_2,\ldots$ in $\h^3$ there is an almost proper (hence complete) conformal $\CMC$ immersion $\psi: M\to\h^3$ with $\psi(a_j)=c_j$ for $j=1,2,\ldots$ 
In particular, there is an almost proper (hence complete) conformal  $\CMC$ immersion $\psi: M\to\h^3$ with everywhere dense image: $\overline{\psi(M)}=\h^3$.
\end{corollary}
Almost proper ones are in some sense the best class of immersions $M\to \h^3$ that can hit an arbitrary countable subset of $\h^3$.
It remains an open question whether every open Riemann surface $M$ admits a proper conformal $\CMC$ immersion $M\to\h^3$ \cite[Problem 1]{AlarconForstneric2015MA}. Every such $M$ is known to properly conformally immerse in $\R^3$ as a minimal surface \cite{AlarconLopez2012JDG}; see \cite[\textsection 3.10]{AlarconForstnericLopez2021Book} for further discussion.
As far as we are aware, Corollary \ref{co:intro-dense} provides the first known examples of a complete densely immersed Bryant surface in $\h^3$. See \cite{AlarconCastroInfantes2018GT} and the references therein for similar existence results for complete dense minimal surfaces in Euclidean space.

It follows from the Lawson correspondence \cite{Lawson1970AM} that every simply connected $\CMC$ surface in $\h^3$ is isometric to a minimal surface in $\r^3$, and vice versa. The intrinsic local theory of Bryant surfaces is thus equivalent to that of minimal surfaces in $\r^3$. If this was originally the main reason to study $\CMC$ surfaces in $\h^3$ ---but see Donaldson \cite[\textsection3.3]{Donaldson1992JGP} and Hertrich-Jeromin, Musso, and Nicolodi \cite{Hertrich-JerominMussoNicolodi2001} for different motivations---, the topic became an active focus of interest in its own right which counts with a sizable literature.
On the other hand, the global theory of Bryant surfaces differs substantially from that of minimal surfaces in $\r^3$; see Rossman \cite{Rossman01}, Rosenberg \cite{Rosenberg2002}, and the references therein for information in this direction. 
Furthermore, the period problem for Bryant surfaces cannot be solved by classical potential theory, and hence constructing Bryant surfaces with nontrivial topology via conformal representation is a more difficult task than for minimal surfaces (see G\'alvez and Mira \cite[p.\ 458]{GalvezMira2005} and Pirola \cite{Pirola2007AJM} for a discussion of this fact). 
This has not prevented the development of a number of different constructions of Bryant surfaces (see \cite{UmeharaYamada1992APO,UmeharaYamada1993AM,UmeharaYamada1996Ann,RUY97,HauswirthRoitmanRosenberg02,Bobenk03, PaPi04,RUY04,GalvezMira2005,TenenblatWang2009,LRTT12} for a few instances); nevertheless, only few open Riemann surfaces were known to carry a complete conformal CMC-1 immersion into $\h^3$ until now.

A way to overcome the aforementioned difficulty is to make use of a projection $\c^2\times\c^*\to\h^3$, discovered by Mart\'in, Umehara, and Yamada in \cite{MartinUmeharaYamada2009CVPDE}, which takes holomorphic null curves in $\c^2\times\c^*$ (see the next paragraph) into conformal $\CMC$ immersions into $\h^3$; see Section \ref{sec:SL2C} for the details. (As it is customary, we denote $\c^*=\c\setminus\{0\}$.) This strategy has been exploited by Alarc\'on,  Forstneri\v c, and L\'opez \cite{AlarconLopez2013MA,AlarconForstneric2015MA} in order to produce Bryant surfaces with arbitrary topological type or conformally equivalent to any bordered Riemann surface and enjoying various global properties. In this paper we delve into this direction and show, in particular, that this approach can be adapted to construct complete Bryant surfaces with arbitrary complex structure and the additional conditions in Theorem \ref{th:intro-Bryant}.

We now present our results on holomorphic null curves in $\c^2\times\c^*$ leading to the proof of Theorem \ref{th:intro-Bryant}. Let $M$ be an open Riemann surface. A {\em holomorphic null curve} $M\to\c^n$ $(n\ge 3)$ is a holomorphic immersion $X=(X_1,\ldots,X_n): M\to\c^n$ which is directed by the null quadric   
\begin{equation}\label{def:nullquadric}
	\mathbf{A} = \big\{  z = (z_1,\ldots,z_n) \in \C^n:\; z_1^2 + \ldots + z_n^2=0 \big\},
\end{equation}
in the sense that the derivative $X'=(X_1',\ldots,X_n')$ with respect to any local holomorphic coordinate on $M$ has range in ${\bf A}_*={\bf A}\setminus\{0\}$; equivalently, the differential $dX=(dX_1,\ldots,dX_n)$ vanishes nowhere on $M$ and satisfies $\sum_{j=1}^n (dX_j)^2=0$ everywhere on $M$.
The real and imaginary part of a holomorphic null curve $M\to\c^n$ are conformal minimal immersions $M\to\r^n$. Conversely, every conformal minimal immersion $M\to\r^n$ is locally on each simply connected domain of $M$ the real part of a holomorphic null immersion into $\c^n$. This connection has strongly influenced the theory of minimal surfaces, supplying the field with powerful tools coming from complex analysis and the theory of Riemann surfaces, as well as from modern Oka theory (see the monographs by Forstneri\v c \cite{Forstneric2017E,Forstneric2023IM} for background on this subject). We refer to \cite{AlarconForstnericLopez2021Book} for a recent monograph on minimal surfaces in $\r^n$ and holomorphic null curves in $\c^n$  from a complex analytic viewpoint.

The second main result of independent interest in this paper is the following Runge type approximation theorem with Weierstrass-Florack jet interpolation for complete holomorphic null curves in $\c^2\times\c^*$ (see the more precise Theorem \ref{th::Runge}). 
%
%
\begin{theorem}\label{th:intro-C2xC*}
Let $M$ be an open Riemann surface, $K\subset M$ be a compact subset such that $M\setminus K$ has no relatively compact connected components, and $\Lambda\subset M$ be a closed discrete subset. If $X: U\to \c^2\times\c^*$ is a holomorphic null curve on an open neighborhood $U\subset M$ of $K\cup\Lambda$, then for any number $\varepsilon>0$ and any map $k:\Lambda\to\n$ there is a complete holomorphic null curve $\wt X: M\to\c^2\times\c^*$
such that $|\wt X-X|<\varepsilon$ everywhere on $K$, and $\wt X-X$ vanishes to order $k(p)$ at $p$ for every point $p\in\Lambda$. Furthermore, $\wt X$ can be chosen injective provided that $X|_\Lambda$ is injective.

In fact, there is a holomorphic null curve $\wt X=(\wt X_1,\wt X_2,\wt X_3): M\to\c^2\times\c^*$ satisfying the above properties such that $(\wt X_1,\wt X_2): M\to\c^2$ is an almost proper map. 
\end{theorem}
The analogue of Theorem \ref{th:intro-C2xC*} for holomorphic null curves in $\C^3$ is already known except for the final assertion about the almost properness condition; see  \cite{AlarconLopez2012JDG,AlarconForstneric2014IM,AlarconCastro-Infantes2019APDE}. The main new point is that we are now working not in the whole space $\C^3$ but in a considerably smaller target: $\c^2\times\c^*$. The proof, given in Section \ref{sec:Runge} as application of the main technical lemma which we establish in Section \ref{sec:lemma} (see Lemma \ref{lemma::ind}), broadly follows the approach in the mentioned sources based on the control of periods via the use of period dominating sprays of holomorphic maps into the punctured null quadric ${\bf A}_*={\bf A}\setminus\{0\}$ \eqref{def:nullquadric}, but we proceed with additional precision and introduce a new idea that enables us to prevent the third component function from vanishing.
We refer the reader to Sections \ref{sec:Runge} and \ref{sec:Carleman} for more precise statements including also Mergelyan and Carleman type approximation, as well as for analogous results for holomorphic null curves in $\C^3$ with prescribed zero set of the third component function (see Theorems \ref{th::Runge}  and \ref{th::Carleman}). The latter is precisely what shall allow us to guarantee the additional control on the asymptotic behavior of the conformal $\CMC$ immersion $\psi$ in Theorem \ref{th:intro-Bryant}, conditions \ref{prop2} and \ref{prop3}. A minor modification of the proof gives the analogues of Theorem \ref{th:intro-C2xC*} for holomorphic null curves in $\C^l\times(\C^*)^n$ for arbitrary integers $l\ge 2$ and $n\ge 1$. 

Our results on holomorphic null curves in $\C^3$ in Section \ref{sec:Runge} also have applications to holomorphic null curves in the special linear group $SL_2(\c)$ and to spacelike surfaces of constant mean curvature $1$ with admissible singularities in de Sitter 3-space $\s^3_1$ (called $\CMC$ {\em faces}). In particular, we obtain analogues of Theorem \ref{th:intro-Bryant} for these families of surfaces; we discuss these results in Section \ref{sec:SL2C}.
On the other hand, similar existence, approximation, and interpolation results for almost proper, injectively immersed complex curves in $\c^2$ have recently been obtained by Alarc\'on and Forstneri\v c in \cite{AlarconForstneric2023}, with a method completely different from the one in this paper.

%
%
\section{Semi-global approximation for holomorphic null curves in $\C^2\times\C^*$}\label{sec:lemma}

\noindent We shall denote $\igot = \sqrt{-1}\in\C$ and by $| \cdot |$, $\dist ( \cdot, \cdot)$, and $\length(\cdot)$ the Euclidean norm, distance, and length on $\C^n$ $(n \in \n)$, respectively. 
Given a set $K$ in a complex manifold we denote by $\Oscr(K)$ the space of holomorphic functions on $K$; i.e., extending holomorphically to an open neighborhood of $K$. If $K$ is compact we denote by $\Ascr^r(K)$ the space of all $\Cscr^r$ functions $K\to\C$ which are holomorphic in the interior $\mathring K$ of $K$, and write $\Ascr(K)$ for $\Ascr^0 (K)$. Likewise, we define the spaces $\Oscr(K,N)$ and $\Ascr^r(K,N)$ for maps into an arbitrary complex manifold $N$; nevertheless we shall omit the target from the notation when it is clear from the context.
For a $\Cscr^r$ map $f\colon K\to \C^n$, we denote by $\| f \|_{r,K}$ the $\Cscr^r$ maximum norm of $f$ on $K$.

Given a meromorphic function $h \colon M \to \C$ on an open Riemann surface $M$, and a point $p \in M$, we denote by
\begin{equation}\label{eq:orden}
	\ord_p (h) \in \Z
\end{equation}
the only integer such that $z^{-\ord_p (h)} h$ is holomorphic and nonzero at $p$, where $z$ is any local holomorphic coordinate on $M$ with $z(p)=0$.

Let us recall the type of sets and maps we shall use for the Mergelyan approximation for holomorphic null curves; see \cite[Def.\ 1.12.9 and 3.1.3]{AlarconForstnericLopez2021Book}.
%
%
\begin{definition}\label{def:admissible}
An \emph{admissible set} in a smooth surface $M$ is a compact set of the form $S = K \cup \Gamma \subset M $, where $K$ is a (possibly empty) finite union of pairwise disjoint compact domains with piecewise $\Cscr^1$ boundaries in $M$ and $\Gamma= S \setminus \mathring K =\overline{S\setminus K}$ is a (possibly empty) finite union of pairwise disjoint smooth Jordan arcs and closed Jordan curves meeting $K$ only at their endpoints (if at all) and such that their intersections with the boundary $bK$ of $K$ are transverse.
\end{definition}
\begin{definition}\label{def:gnc}
	 Let $S = K \cup \Gamma$ be an admissible set in a Riemann surface $M$ and $\theta$  be a nowhere vanishing holomorphic $1$-form on a neighborhood of $S$. A \emph{generalized null curve} $S \to \C^n$ $(n \geq 3)$ of class $\Ascr^r (S)$ $(r \in \n)$ is a pair $(X,f \theta$) with $X \in \Ascr^r ( S, \C^n)$ and $f \in \Ascr^{r-1} ( S, \mathbf{A}_*)$ (see \eqref{def:nullquadric}) such that $f \theta = dX $ holds on $K$ (hence $X \colon \mathring K \to \C^n$ is a holomorphic null curve) and for any smooth path $\alpha$ in $M$ parameterizing a connected component of $\Gamma$ we
	 have $ \alpha^* (f \theta) = \alpha^* (d X)= d ( X \circ \alpha)$.
\end{definition}

Recall that a compact subset $K$ of an open Riemann surface $M$ is said to be \emph{Runge} (also {\em holomorphically convex} or {\em $\Oscr(M)$-convex}) if every function in $\Ascr (K)$ may be approximated uniformly on $K$ by functions in $\Oscr(M)$. By the Runge-Mergelyan theorem  \cite[Theorem 5]{FFW18} this is equivalent to $M\setminus K$ having no relatively compact connected components in $M$ (i.e., holes). 
A map $f\colon M \to \C^n$ is {\em flat} if $f(M)$ is contained in an affine complex line in $\C^n$, and {\em nonflat} otherwise.

In this section we establish the following semiglobal approximation result with jet interpolation for holomorphic null curves in $\C^3$ with control on the zero set of the third component function.
%
%
\begin{lemma} \label{lemma::ind}
Let $M$ be an open Riemann surface, $\theta$ be a nowhere vanishing holomorphic $1$-form on $M$, $S\subset M$ be a Runge admissible  subset,
$L\subset M$ be a smoothly bounded compact domain such that $S\subset\mathring L$, and $\Lambda \subset \mathring{S}$ be a finite subset.
If $(X = (X_1,X_2,X_3), f\theta)$ is a generalized null curve  $S\to \C^3 $  of class $\Ascr^r(S)$ $(r \geq 1)$ such that 
\begin{equation}\label{eq:X3Lambda}
	X_3^{-1} (0) \subset \Lambda,
\end{equation} 
then for any numbers $\varepsilon>0$ and $k\in\n$ there is a nonflat holomorphic null curve $\wt{X}=(\wt X_1,\wt X_2,\wt X_3) \colon L \to \C^3 $ satisfying the following conditions: 
	\begin{enumerate}[label= \rm (\alph*)]
		\item \label{conda}  $\| \wt{X} - X \|_{r,S} < \varepsilon$.
		\item \label{condb} $\wt{X} - X$ vanishes to order $k$ at every point of $ \Lambda$.
		\item \label{condc} $ \wt X_3^{-1}(0) = X_3^{-1}(0)\subset\Lambda$.
	\end{enumerate}
\end{lemma}
\begin{proof}
Up to passing to a larger $L$ if necessary, we assume that $L\subset M$ is connected and Runge. We also assume without loss of generality that 
\begin{equation}\label{eq:k>}
	k>\max\{\ord_p(X_3): p\in X_3^{-1}(0)\};
\end{equation}
see \eqref{eq:orden}.
We shall first prove the following special case of the lemma.
%
%
\begin{claim}\label{cl:lemma}
The conclusion of the lemma holds if $S$ is a connected smoothly bounded compact domain which is a strong deformation retract of $L$ and $X\colon S\to \C^3$ is a nonflat holomorphic null curve.
\end{claim}
\begin{proof}	
In the first step of the proof we shall follow the construction in \cite[Sec.\ 3.2]{AlarconForstnericLopez2021Book} in order to embed the map $f= d X/\theta\in\Oscr(S,\mathbf{A}_*)$ as the core of a suitable period dominating spray of maps in $\Oscr(S,\mathbf{A}_*)$.
	For this, we fix a point $p_0\in\mathring S\setminus \Lambda\neq\varnothing$ and a finite family $ \{ C_j \}_{j = 1}^n$ of smooth Jordan arcs and curves in $\mathring S$ satisfying the following properties (see \cite{FarkasKra1992} or \cite[Lemma 1.12.10]{AlarconForstnericLopez2021Book}):
\begin{enumerate}[label=$\bullet$]
	\item $C_i\cap C_j=\{p_0\}$ for any $i\neq j\in\{1,\ldots,n\}$.
	\item $\{ C_1, \ldots, C_{\mu} \}$, where $\mu\in\{0,\ldots,n\}$ is the cardinal of $\Lambda$, are smooth Jordan arcs with the initial point $p_0$ and the final point in $\Lambda$.
	\item  $\{ C_{\mu + 1}, \ldots, C_n\}$ are smooth Jordan curves which together span the first homology group $\Hcal_1 ( S, \z)=\Hcal_1 ( L, \z)\cong \Z^{n-\mu}$.
	\item The union $\Ccal = \bigcup_{j=1}^n C_j$ is a Runge compact set in $L$, and $\bigcup_{j=\mu+1}^n C_j\subset \Ccal$ is a deformation retract of $S$, hence of $L$.
\end{enumerate}
The period map $\Pscr= ( \Pscr_1,\ldots, \Pscr_n) \colon \Cscr( \Ccal , \C^3 ) \to (\C^3)^n = \c^{3n}$ associated to the family $\{C_j \}_{j=1}^n$ is defined by the expression
\[
	\Cscr( \Ccal , \C^3 )\ni h\longmapsto \Pscr_j (h) = \int_{C_j} h \theta \in \C^3,\quad j=1,\ldots,n.
\]
The holomorphic map $f \colon S \to \mathbf{A}_*\subset \C^3$ is nonflat by nonflatness of $X$, hence $f|_{C_j} \colon C_j \to \mathbf{A}_*$ is nonflat too for every $j =1,\ldots , n$. This and the fact that $\Ccal$ is a Runge compact subset of $M$ allow to find holomorphic vector fields $V_{j,l}$ on $\C^3$ tangent to $\mathbf{A}$ with flows  $\phi^{j,l}_{\zeta}$ ($\zeta \in \c$), holomorphic functions $h_{j,l} \colon L \to \C$, $(j,l)\in\{1,\ldots , n\}\times\{1,2,3\}$, a holomorphic function $g \colon L \to \C$ vanishing to order $k$ everywhere on $\Lambda \subset \mathring S$, and an open neighborhood $U$ of $0 \in \C^{3n}$ such that the holomorphic map $\Phi_f \colon S \times U \to \mathbf{A}_*$ given by 
\begin{equation}\label{eq::spraydef}
	\Phi_f (p, t) = \phi^{1,1}_{g h_{1,1} (p) t_{1,1}} \circ \ldots  \circ \phi^{n,3}_{g h_{n,3} (p) t_{n,3}} ( f(p) )  \in \mathbf{A}_*, \quad t=(t_{j,l})\in U \subset (\C^3)^n,
\end{equation}
is a $\Pscr$-dominating spray of maps with the core $f$, meaning that the following properties are satisfied (see \cite[Lemma 3.2.1]{AlarconForstnericLopez2021Book}):
\begin{enumerate}[label=(\Roman*)]
	
	\item \label{coref} $\Phi_f ( \cdot, 0) = f$.
	\item \label{intf}$f -  \Phi_{ f}(\cdot, t) $ vanishes to order $k$ everywhere on $\Lambda$ for every $t \in U$.
	\item \label{eq::domination} The derivative at $t=0$ 
	\[
	\frac{\partial}{\partial t} \bigg|_{t= 0} \Pscr(  \Phi_{f} ( \cdot, t) ) \colon (\C^3)^n \to ( \C^3)^n \;  \text{ is an isomorphism.}
	\]	
	\newcounter{Mayus4}\setcounter{Mayus4}{\value{enumi}}
\end{enumerate}

In the second step of the proof we shall replace $f$ in \eqref{eq::spraydef} by a suitable holomorphic map $\hat f=(\hat f_1,\hat f_2, \hat f_3)\in\Oscr(L,\mathbf{A}_*)$; in particular, we shall ensure that $\hat f_3\theta = d \wh X_3$ for a holomorphic function $\wh X_3:L\to\C$ which is close to $X_3$ on $S$ and whose zero set on $L$ equals the one of $X_3$ on $S$.  This is the key ingredient to ensure condition \ref{condc} in the statement of the lemma.
For this, call $f=(f_1,f_2,f_3)$ and $A=f_3^{-1}(0)\subset S$; note that $f_3$ is not identically zero since $X$ is nonflat, hence $A$ is finite. Up to passing to larger $S$ and $k$ if necessary, we assume that $A\subset \mathring S$ and 
\begin{equation}\label{eq:k>>}
	k>\max\{\ord_p (f_3): p\in A=f_3^{-1}(0)\};
\end{equation}	
cf.\ \eqref{eq:k>}.
Fix $\delta>0$ to be specified later. Since $\c^*$ is an Oka manifold (see
\cite[Corollary 1.13.9]{AlarconForstnericLopez2021Book}),
	the Runge approximation theorem with jet interpolation  for maps into Oka manifolds (see \cite[Theorem 1.13.3]{AlarconForstnericLopez2021Book} or \cite[Theorem 5.4.4]{Forstneric2017E}) furnishes us with a holomorphic function $\wh X_3 \colon L \to \c$ such that:
	\begin{enumerate}[label= \rm (\roman*)]
		\item \label{whX3a} $\| \wh X_3  - X_3 \|_{r,S} <  \delta $.
		\item \label{whX3i} $\wh X_3- X_3$ vanishes to order $k$ everywhere on $\Lambda\cup A \cup\{p_0\}$. 
		\item \label{whX3z} $\wh X_3^{-1} (0) = X_3^{-1}(0)\subset \Lambda\subset \mathring S$.
		\newcounter{Mayus3}\setcounter{Mayus3}{\value{enumi}}
	\end{enumerate}
Indeed, let $\tau\colon L\to\C$ be a holomorphic function whose zero divisor on $L$ equals that of $X_3$ on $S$; hence $\tau^{-1}(0)=X_3^{-1}(0)$ and $X_3/\tau$ is holomorphic and nowhere vanishing on $S$. Such a $\tau$ exists by the classical Weierstrass theorem on open Riemann surfaces \cite[Theorem 1.12.13]{AlarconForstnericLopez2021Book}. By the aforementioned \cite[Theorem 1.13.3]{AlarconForstnericLopez2021Book}, we may approximate $X_3/\tau\colon S\to\C^*$ uniformly on $S$ by a holomorphic function $\sigma\colon L\to\C^*$ which agrees with $X_3/\tau$ to order $k$ everywhere on $\Lambda\cup A \cup \{p_0\}$. The function $\wh X_3=\sigma\tau\colon L\to\C$ then satisfies the required conditions if the approximation of $X_3/\tau$ by $\sigma$ is close enough.
	
Consider the holomorphic function
\begin{equation}\label{eq:f3}
 	\hat  f_3 := d \wh X_3/\theta \in\Oscr(L).
\end{equation}
Properties \ref{whX3a} and \ref{whX3i} ensure that:
	\begin{enumerate}[label= \rm (\roman*)]
		\setcounter{enumi}{\value{Mayus3}}
		\item \label{whf3a} $\| \hat f_3  - f_3 \|_{r-1,S} < \delta$.
		\item \label{whf3i} $\hat f_3-f_3$ vanishes to order $k-1$ everywhere on $\Lambda\cup A $. 
\setcounter{Mayus3}{\value{enumi}}
\end{enumerate}
	Assuming that $\delta>0$ is chosen sufficiently small, \ref{whf3a}, \ref{whf3i}, \eqref{eq:k>>}, and Hurwitz theorem (see \cite{Hurwitz1985} or \cite[VII.2.5, p.148]{Conway1973}) guarantee that 
	\begin{enumerate}[label= \rm (\roman*)]
		\setcounter{enumi}{\value{Mayus3}}
		\item \label{whf3z} the zero divisor of $\hat f_3$ on $S$ equals that of $f_3$. In particular, we have that $\hat f_3^{-1}(0)\cap S=f_3^{-1}(0)=A \subset\mathring S$. 
\setcounter{Mayus3}{\value{enumi}}
\end{enumerate}		
Set $\eta := f_1 + \igot f_2 \in \Oscr(S)$. Since $f$ has range in $\mathbf{A}_*$, we have that
	\begin{equation}\label{eq:f1f2}
	\frac{f_3^2}{\eta}=-f_1+\igot f_2\in\Oscr(S),\quad 
	f_1 =   \frac{1}{2} \left( \eta - \frac{f_3^2}{\eta }\right),
	\quad\text{and}\quad
	 f_2 = -\frac{\igot}{2} \left( \eta + \frac{f_3^2}{\eta } \right).
\end{equation}
Moreover, $\eta$ and $f_3^2/\eta$ have no common zeros.
	As above, the Runge theorem with jet interpolation in \cite[Theorem 1.13.3]{AlarconForstnericLopez2021Book} provides a function $\hat \eta\in\Oscr(L)$ such that:
	\begin{enumerate}[label= \rm (\roman*)]
		\setcounter{enumi}{\value{Mayus3}}
		\item \label{whea} $\| \hat \eta-\eta \|_{r-1, S}<\delta$.
		\item \label{whei} $\hat \eta-\eta$ vanishes to order $k-1$ at every point in $\Lambda\cup A$.
		\item \label{whez} The divisor of $\hat \eta$ on $L$ equals that of $\eta$ on $S$.
		\setcounter{Mayus3}{\value{enumi}}
	\end{enumerate}
	By \ref{whf3z}, \ref{whez}, and \eqref{eq:f1f2}, the function $\hat f_3^2/\hat \eta$ is holomorphic on $L$, and hence so are
	\begin{equation}\label{eq:whf1f2}
		 \hat f_1 =   \frac{1}{2} \left( \hat \eta - \frac{ \hat f_3^2}{ \hat \eta }\right) \quad \text{and}\quad \hat f_2 = -\frac{\igot}{2} \left( \hat \eta + \frac{\hat f_3^2}{\hat \eta } \right). 
	\end{equation}
Moreover, $\hat f_3^2/\hat \eta$ and $f_3^2/\eta$ have the same divisor on $S$, and hence, in view of \ref{whez}, $\hat \eta$ and $\hat f_3^2/\hat \eta$ have no common zeros on $L$ (recall that $\eta$ and $f_3^2/\eta$ have no common zeros on $S$ and $\hat \eta$ vanishes nowhere on $L\setminus S$). This shows that the holomorphic map $\hat f:=(\hat f_1, \hat f_2,\hat f_3)\colon L\to \C^3$ has range in $\mathbf{A}_*$.
It enjoys the following properties: 
	\begin{enumerate}[label= \rm (\roman*)]
		\setcounter{enumi}{\value{Mayus3}}
		\item \label{whfa} $\| \hat f-f \|_{r-1, S}<\delta$.
		\item \label{whfi} $\hat f-f$ vanishes to order $k-1$ at every point in $\Lambda$.
		\item \label{whfnf} $\hat f \colon L \to \mathbf{A}_*$ is nonflat.
		\setcounter{Mayus3}{\value{enumi}}
	\end{enumerate}
Indeed, \ref{whfa} follows from \ref{whf3a}, \eqref{eq:f1f2},  \ref{whea}, and  \eqref{eq:whf1f2} whenever the approximations are close enough, while \ref{whfi} follows from \ref{whf3i}, \eqref{eq:f1f2},  \ref{whei}, and  \eqref{eq:whf1f2}.	Property \ref{whfnf} holds provided we choose $\wh f$ close enough to $f$ on $S$; recall that $L$ is connected and $X$ is nonflat on $S$, hence so is $f$.

Properties \ref{whX3i} and \eqref{eq:f3} show that  $\int_{C_j} \hat f_3 \theta=  \int_{C_j} f_3 \theta$ for all $j=1,\ldots,n$. However, the analogous conditions for $\hat f_1$ and $\hat f_2$ may not happen, and hence $\Pscr(\hat f)$ need not agree with $\Pscr(f)$. To solve this issue we 
replace $f$ by $\hat f$ in the spray \eqref{eq::spraydef}. This way, if the approximation of $f$ by $\hat f$ in \ref{whfa} is close enough and up to passing to a smaller neighborhood $U$ of $0$ in $\C^{3n}$, we obtain a well defined holomorphic map $\Phi_{\hat f} \colon L \times U \to \mathbf{A}_*$, given by
\[
	\Phi_{\hat f} (p, t) = \phi^{1,1}_{g h_{1,1} (p) t_{1,1}} \circ \ldots  \circ \phi^{n,3}_{g h_{n,3} (p) t_{n,3}} (  \hat f (p) )  \in \mathbf{A}_*,  \quad t=(t_{j,l})\in U \subset (\C^3)^n
\]
(recall that $h_{i,j}$ and $g$ are defined on $L$), satisfying the following properties:
\begin{enumerate}[label= \rm (\Roman*)]
		\setcounter{enumi}{\value{Mayus4}}
		\item \label{corewhf} $\Phi_{\hat f} (\cdot , 0) = \hat f$ (i.e., $\hat f$ is the core of the spray).
		\item \label{intwhf} $\hat f - \Phi_{\hat f} ( \cdot, t)$ vanishes to order $k$ everywhere on $\Lambda$ for every $t \in U$.
		\item \label{eq::dominationL} $\Phi_{\hat f}$ is $\Pscr$-dominating in the sense that the derivative at $t= 0$
		\[
		\frac{\partial}{\partial t} \bigg|_{t= 0} \Pscr (  \Phi_{\hat f } ( \cdot, t)  ) \colon ( \C^3 )^n \to ( \C^3 )^n \;  \text{ is an isomorphism.}
		\]
	\setcounter{Mayus4}{\value{enumi}}
\end{enumerate}
Moreover, assuming that  
the approximation in \ref{whfa} is close enough, properties \ref{coref} and \ref{corewhf} ensure that
\begin{enumerate}[label=(\Roman*)]
		\setcounter{enumi}{\value{Mayus4}}
	\item \label{approxsprays} $\|  \Phi_f  - \Phi_{\hat f}    \|_{r-1, S\times U} < \delta$.
\end{enumerate} 
In view of \ref{coref}, \ref{eq::domination}, \ref{corewhf}, \ref{eq::dominationL}, \ref{approxsprays}, and assuming the approximation in \ref{whfa} is close enough, the inverse function theorem gives us a point $\zeta_0 \in U\subset \c^{3n}$ close to $0$ such that
	\begin{equation}\label{eq::periodsolved}
		\Pscr \big(  \Phi_{\hat f} ( \cdot, \zeta_0)  \big) = \Pscr( f).
	\end{equation}
Since $f\theta= d X$ is exact, this implies that $\Phi_{\hat f} ( \cdot, \zeta_0)\theta$ is exact as well, and hence we obtain a holomorphic map $\wt X=(\wt X_1,\wt X_2,\wt X_3)\colon L\to \C^3 $ given by
	\begin{equation}\label{eq:tildeX}
		\wt X (p) =  X(p_0)+\int_{p_0}^p   \Phi_{\hat  f}  ( \cdot , \zeta_0) \theta \in \C^3 , \quad p \in L.
	\end{equation} 
	We claim that $\wt X$ satisfies the conclusion of the lemma. Indeed, since  $\Phi_{\hat f}  ( \cdot , \zeta_0)$ has range in $\mathbf{A}_*$, we have that $\wt X$ is a holomorphic null immersion. 
	Condition \ref{conda} in the lemma follows from \ref{approxsprays} and the fact that $\zeta_0$ can be chosen as close to $0$ as desired, provided that $\delta>0$ is chosen sufficiently small.
	 This and the nonflatness of $X$ also guarantee that $\wt X$ is nonflat. 
	 For \ref{condb}, note that \eqref{eq::periodsolved} guarantees that $\wt X(p)=X(p)$ for all $p\in \Lambda$, while the higher order interpolation follows from \ref{intwhf} and \ref{whfi}.
	 Finally, to check \ref{condc} it suffices to see that $\wt X_3^{-1}(0)=\wh X_3^{-1}(0)$; this and  \ref{whX3z} complete the task. Indeed, $\wh X_3^{-1}(0)= X_3^{-1}(0) \subset \Lambda\subset\mathring S$ by \ref{whX3z}. Therefore, \eqref{eq:X3Lambda} and \ref{condb} ensure that $\wh X_3^{-1}(0) \subset \wt X_3^{-1}(0)$. This, \eqref{eq:k>}, \ref{whX3i}, \eqref{eq:f3}, \ref{intwhf}, and \eqref{eq:tildeX} show that the divisor of $\wt X_3$ on $L$ is greater than or equal to the one of $\wh X_3$. For the other inclusion, \ref{whX3i} gives that $\wh X_3(p_0)=X_3(p_0)$, which together with \eqref{eq:f3}, \ref{corewhf}, and \eqref{eq:tildeX} guarantee that $\wt X_3$ can be chosen arbitrarily close to $\wh X_3$ uniformly on $L$. Since the zeros of $\wh X_3$ lie in $\Lambda\subset \mathring L$, Hurwitz theorem implies that $\wt X_3^{-1}(0)=\wh X_3^{-1}(0)$. 
	 \end{proof}	
	
For the general case of the lemma, the Runge-Mergelyan approximation theorem with jet interpolation for generalized null curves (see \cite[Theorem 1.2]{AlarconCastro-Infantes2019APDE} or \cite[Theorem 3.6.2]{AlarconForstnericLopez2021Book}) enables us to approximate $X$ in the $\Cscr^r(S)$-norm by a nonflat holomorphic null curve $Y=(Y_1,Y_2,Y_3):L\to \C^3$ such that $Y-X$ vanishes to order $k$ everywhere on $\Lambda$. Note that $Y$ formally satisfies conditions \ref{conda} and \ref{condb}  but it need not meet \ref{condc}. By \ref{condb}, \eqref{eq:X3Lambda}, and \eqref{eq:k>}, and assuming that the approximation of $X$ by $Y$ in the $\Cscr^r(S)$-norm  is close enough, Hurwitz theorem guarantees that $Y_3^{-1}(0)\cap S=X_3^{-1}(0)\subset\Lambda\subset\mathring S$, while the set $Q:=Y_3^{-1}(0)\cap \mathring L\setminus S$ is finite by the identity principle. Since the compact set $S$ is Runge, we can choose a family of pairwise disjoint smooth Jordan arcs $\varsigma_q\subset L\setminus S$, $q\in Q$, such that $\varsigma_q$ has $q$ as an endpoint, the other endpoint of $\varsigma_q$ lies in $bL$, $\varsigma_q$ intersects $bL$ transversely there, and $\varsigma_q$ is otherwise disjoint from $bL$. Set $\varsigma=\bigcup_{q\in Q}\varsigma_q$ and note that $\mathring L\setminus \varsigma$ is an open connected domain, $S\subset \mathring L\setminus \varsigma$, and $(Y_3|_{\mathring L\setminus \varsigma})^{-1}(0)=X_3^{-1}(0)$. Let $S'\subset\mathring L\setminus \varsigma$ be a connected smoothly bounded Runge compact domain such that $S\subset\mathring S'$ and $S'$ is a strong deformation retract of $L$. Since the connected compact domain $L\subset M$ is Runge by the initial assumption, it follows that $S'$ is Runge in $M$ as well. Moreover, we have that $(Y_3|_{S'})^{-1}(0)=X_3^{-1}(0)\subset\Lambda\subset\mathring S$. This reduces the proof of Lemma \ref{lemma::ind} to the special case in Claim \ref{cl:lemma}.
\end{proof}

%
%
\section{A Runge theorem for almost proper null curves in $\C^2\times\C^*$}\label{sec:Runge}

\noindent In this section we prove the following Runge-Mergelyan approximation theorem with jet interpolation for almost proper holomorphic null curves in $\C^3$ with control on the zero set of the third component function. Recall Definitions \ref{def:admissible} and \ref{def:gnc}.
\begin{theorem}\label{th::Runge} 
	Let  $M$ be an open Riemann surface, $\theta$ be a nowhere vanishing holomorphic $1$-form on $M$, $S \subset M$ be a Runge admissible subset, and $\Lambda\subset M$ be a closed discrete subset.
	Also let $(X=(X_1,X_2,X_3),f \theta)$ be a generalized null curve $S \to \C^3$ of class $\Ascr^r(S)$ $(r\ge 1)$ which is a holomorphic null curve on a neighborhood of $\Lambda$, and assume that 
	\begin{equation}\label{eq:Lambda}
		X_3^{-1} (0) \subset \Lambda.
	\end{equation}
	Given a map $k\colon \Lambda \to \n$
	and a number $\varepsilon>0$, there exists a holomorphic null curve  $\wt X=(\wt X_1,\wt X_2,\wt X_3) \colon M \to \C^3$ such that:
	\begin{enumerate}[label={\rm (\alph*)}]
		\item  \label{rconda} $\| \wt X - X \|_{r, S}   < \varepsilon$.
		\item \label{rcondb} $\wt X - X$ vanishes to order $k (p)$ at every point $p \in \Lambda$.
		\item  \label{rcondc}$\wt X_3^{-1}(0)=X_3^{-1}(0)\subset\Lambda$. 
		\item  \label{rcondd}Both $(\wt X_1,\wt X_2) \colon M \to \C^2$ and $(1, \wt X_1, \wt X_2)/\wt X_3 \colon M \setminus X_3^{-1}(0) \to \C^3$ are almost proper maps. In particular, $\wt X\colon M\to\C^3$ is almost proper as well.
	\end{enumerate}
	Furthermore, if $X|_\Lambda \colon \Lambda \to \C^3$ is injective then there is an injective holomorphic null curve $\wt X \colon M \to \C^3$ with the above properties. 
\end{theorem}
Theorem \ref{th:intro-C2xC*} follows as an immediate corollary of Theorem \ref{th::Runge}. Indeed, property \ref{rcondc} ensures that $\wt X$ has range in $\C^2 \times \C^*$ whenever $X$ does so. On the other hand, recall that a path $\gamma \colon [0,1) \to M$ on a smooth surface $M$ is  {\em divergent} if for any compact subset  $K \subset M$ there exists $t_0 \in [0, 1)$ such that $ \gamma (t) \cap K = \varnothing$  for all $t >t_0$; i.e., if $\gamma$ is a proper map. Given a smooth immersion $X \colon M \to \C^n$ $(n \geq 2)$ we denote by $\dist_X \colon M \times M \to \R_+$ the Riemannian distance function induced on $M$ by the Euclidean metric of $\C^n$ via $X$; i.e.
\[
\dist_X (p,q) := \inf \{  \length ( X (\gamma)) : \gamma \subset M \text{ arc connecting $p$ and $q$}   \},\quad p,q \in M.
\]
The immersion $X \colon M \to \C^n$ is said to be  {\em complete} if $\length (X \circ \gamma) =+ \infty$ for every divergent path $\gamma$ in $M$. Equivalently, if the Riemannian metric on $M$ induced by $\dist_X$ is complete in the classical sense. On the other hand, when $X$ is almost proper then $X\circ\gamma\colon [0,1)\to\C^n$ is unbounded for every divergent path $\gamma$ on $M$; hence almost proper immersions $M\to\C^n$ are complete.

We shall need some preparations before proving the theorem. 
Given a map $f =(f_1,f_2,f_3)\colon W \to \C^3$, we define the function
\begin{equation}\label{eq::defm}
	\mathfrak{m}(f) := \frac1{\sqrt{2}}\max \{   |f_1+\igot f_2|\,,\, |f_1-\igot f_2|   \} \colon W \to [0,+\infty).
\end{equation}
Note that $\mathfrak{m}(f) \leq |(f_1,f_2)| \le | f |$ everywhere in $W$ by the Cauchy-Schwarz inequality.
We shall make use of the following extension of \cite[Lemma 4]{AlarconForstneric2015MA}. 
\begin{lemma}\label{lemma::ap}
	Let $M= \overline M\setminus b\overline M$ be a bordered Riemann surface,  $K\subset M$ be a smoothly bounded compact domain,  $\Lambda \subset M$ be a finite set,  $s >0$ be a number, and  $X = (X_1, X_2, X_3) \colon \overline M \to \C^3$  be a null curve of class  $\Ascr^r(\overline M)$ $(r\ge 1)$, and assume that
		$\mathfrak{m}(X) > s$ on $\overline M \setminus \mathring K$.
	Given numbers $\varepsilon>0$, $\hat s > s$, and $k \in \n$, there exists a null curve $\wh X  = ( \wh X_1, \wh X_2, \wh X_3) \colon \overline M \to \C^3$ of class  $\Ascr^r(\overline M)$ satisfying the following:
	\begin{enumerate}[label=\rm (\roman*)]
		\item \label{L1} $\| \wh X - X \|_{r, K} < \varepsilon$.
		\item \label{L2} $\mathfrak{m} (\wh X) > \hat s $ on $b\overline M$.
		\item  \label{L3} $\mathfrak{m}(\wh X) > s$ on $\overline M \setminus \mathring K$.
		\item \label{L4} $ \| \wh X_3 - X_3 \|_{0, \overline M } < \varepsilon.$
		\item  \label{L5} $\wh X - X $ vanishes to order $k$ everywhere on $\Lambda$.
	\end{enumerate}
\end{lemma}
Recall that a bordered Riemann surface is an open connected
Riemann surface $M$ that is the interior, $M= \overline M\setminus b\overline M$, of a compact one dimensional
complex manifold $\overline M$ with smooth boundary $b\overline M$ consisting of finitely many closed
Jordan curves; such an $\overline M$ is called a compact bordered Riemann surface. 
The only improvements of this lemma with respect to \cite[Lemma 4]{AlarconForstneric2015MA} concern conditions \ref{L4} and \ref{L5}. On the one hand, the interpolation condition \ref{L5} can be granted by following \cite[proof of Lemma 4]{AlarconForstneric2015MA} but applying \cite[Theorem 6.4.2]{AlarconForstnericLopez2021Book} (the Riemann-Hilbert problem for null curves in $\C^3$ with jet interpolation) instead of  
 the more basic result \cite[Theorem 4]{AlarconForstneric2015MA} which does not include interpolation. Likewise, one has to apply Mergelyan theorem with jet interpolation when required. On the other hand, condition \ref{L4} is stronger than condition (L4) in \cite[Lemma 4]{AlarconForstneric2015MA} to the effect that $\|\wh X_3\|_{0, \overline M}<\|X_3\|_{0, \overline M}+\varepsilon$. To guarantee \ref{L4} only very minor modifications of the proof of \cite[Lemma 4]{AlarconForstneric2015MA} are required. In particular, property (b4) in that proof must be replaced by the more accurate condition that $|\langle z,(0,0,1)\rangle-F_3(p_{i,j})|<\epsilon/2$ for all $z\in\lambda_{i,j}$ (the given null curve in  \cite[Lemma 4]{AlarconForstneric2015MA} is denoted by $F=(F_1,F_2,F_3)$). Likewise, conditions (c4) and (e5) in  \cite[proof of Lemma 4]{AlarconForstneric2015MA} must be replaced by $\|F_3^0-F_3\|_{0,\overline M}<\epsilon/2$ and $\|\wh F_3-F_3^0\|_{0,\overline M}<\epsilon/2$, respectively. Besides this, the proof is the same.
 %
 %
\begin{proof}[Proof of Theorem \ref{th::Runge}]
For simplicity of exposition, we assume that $X|_\Lambda$ is injective, the proof is otherwise simpler.
We also assume that $0<\varepsilon<1$ and
\begin{equation}\label{eq::condk}
		k(p) > \ord_p (X_3) \quad \text{for every $p\in X_3^{-1}(0)$};
\end{equation}
see \eqref{eq:Lambda} and recall the notation in \eqref{eq:orden}.
Moreover, arguing as in the final part of the proof of Lemma \ref{lemma::ind}, Hurwitz theorem and 
the Runge-Mergelyan theorem with jet interpolation for generalized null curves in \cite[Theorem 3.6.2]{AlarconForstnericLopez2021Book} enable us to assume that $M_0 := S$ is a connected, smoothly bounded, Runge compact domain and $X$ is a nonflat injective holomorphic null immersion on a neighborhood of  $S\cup\Lambda$ (so an embedding on $M_0$; note that the mentioned approximation result also ensures that holomorphic null curves in $\C^3$ are generically embedded).
Finally, since $X$ is holomorphic on a neighborhood of $\Lambda$, we may also assume that $\Lambda \cap bM_0 = \varnothing$. 	
Let 
\begin{equation}\label{eq:Mncup}
	M_1\Subset M_2\Subset M_3\Subset\cdots \subset \bigcup_{n\in \N} M_n=M
\end{equation}
be an exhaustion of $M$ by connected, smoothly bounded, Runge compact domains such that $M_0 \subset \mathring M_1$ and $\Lambda \cap b M_n = \varnothing$ for every $n \in \n$.  
Choose a closed neighborhood $\Omega\subset M$ of $\Lambda$ whose connected components are smoothly bounded closed discs, each one containing a unique point of $\Lambda$, such that $\Omega \cap  bM_n =\varnothing$ for all $n\ge 0$ and $X \colon M_0\cup \Omega \to \C^3$ is defined as a holomorphic null immersion. 
Choose $\varepsilon_0<\varepsilon/2$ and set $X^0=X|_{M_0}$. We shall inductively construct a sequence of numbers $\varepsilon_n>0$ and nonflat holomorphic null embeddings 
	$X^n = (X^n_1, X^n_2, X^n_3) \colon M_n \to \C^3$
	satisfying the following conditions for all $n\in\N$.
	\begin{enumerate}[label=(\alph*$_n$)]
		\newcounter{an}\setcounter{an}{\value{enumi}}
		\item \label{seqa} $\|X^n - X^{n-1}\|_{r, M_{n-1}}< \varepsilon_{n-1}$. 
		\item \label{seqi} $X^n - X$ vanishes to order $k(p)$ at every point $p \in\Lambda\cap M_n$.
		\item \label{seqX3} $(X^n_3)^{-1} (0) = X_3^{-1}(0)\cap M_n \subset \Lambda$.
		\item \label{seqap} $ \mathfrak{m}(X^n) >n \, (\|X^n_3\|_{0, bM_n} + 1)$ 
		on $bM_n$; see \eqref{eq::defm}.
		
		\item \label{seqep} $0<\varepsilon_{n}<\varepsilon_{n-1}/2$ and if $Y\colon M_n\to\C^3$ is a holomorphic map  that satisfies $\|Y - X^{n}\|_{r, M_n}< 2\varepsilon_{n}$, then $Y$ is a nonflat embedding.
	\end{enumerate}
Provided the existence of sequences $\varepsilon_n >0$ and $X^n \colon M_n \to \C^3$ satisfying \ref{seqa} and \ref{seqep} for all $n \in \n$, there is a limit holomorphic map
	\[
	\wt X:= \lim\limits_{n \to +\infty} X^n\colon M=\bigcup_{n\in \N} M_n\to \C^3
	\]
satisfying 
\begin{equation}\label{eq::n0}
	\|\wt X-X^n\|_{r, M_n}< 2\varepsilon_{n}<\varepsilon\quad \text{for all }n\ge 0.
\end{equation}
Therefore, $\wt X=(\wt X_1,\wt X_2,\wt X_3)\colon M\to\C^3$ is a nonflat injective holomorphic null curve by \ref{seqep} and \eqref{eq:Mncup}.
	We claim that $\wt X$ satisfies the conclusion of the theorem whenever $X^n$ also enjoys properties \ref{seqi}--\ref{seqap} for all $n \in \n$. 
	Indeed, conditions \ref{rconda} and \ref{rcondb} in the statement of the theorem follow from \eqref{eq::n0} and \ref{seqi}. 
	Concerning \ref{rcondc}, we  have
	 $X_3^{-1}(0) \subset \wt X_3^{-1} (0)$ 
	 by \eqref{eq:Lambda} and \ref{rcondb}.
	    To check the converse inclusion $\wt X_3^{-1} (0) \subset X_3^{-1}(0)$, reason by contradiction and assume that there is a point $p\in \wt X_3^{-1} (0) \setminus X_3^{-1}(0)$. Let $U\subset M$ be a smooth open disc neighborhood of $p$ in $M$ such that $X_3\neq 0$ everywhere on $U$. We have that $U\subset M_n$ and $X_3^n\neq 0$ everywhere on $U$  by \ref{seqX3} for all large enough $n\in\N$. Hurwitz theorem then implies that $\wt X_3=0$ everywhere on $U$, hence on $M$, which is a contradiction, recall that $\wt X$ is nonflat. This shows \ref{rcondc}.
 	Finally, let us see \ref{rcondd}. 
 	First, for $(\wt X_1,\wt X_2)$, we use \ref{seqap} and \eqref{eq::n0} to obtain that
	\[
		|(\wt X_1, \wt X_2)| +\varepsilon > |(X_1^n, X_2^n)|  \geq \mathfrak{m}(X^n ) > n  \quad \text{ on  $bM_n$ for all $n \in \n$}.
	\]
 This and \eqref{eq:Mncup} show that $(\wt X_1, \wt X_2) \colon M \to \C^2$ is an almost proper map. Concerning the map $(1, \wt X_1, \wt X_2)/ \wt X_3 \colon M \setminus X_3^{-1}(0) \to \C^3$, which is well defined and holomorphic by property \ref{rcondc}, note that  \ref{seqap}, \eqref{eq::n0}, and the assumption that $\varepsilon<1$ imply that
\begin{equation}\label{eq::normt}
		\Big|   \frac{1}{\wt X_3} (\wt X_1, \wt X_2 ) \Big| > \frac{| \left(X^n_1,  X^n_2 \right)|-\varepsilon}{|X^n_3|+\varepsilon} >  \frac{\mathfrak{m}(X^n)}{| X^n_3 |+1} -1 > n -1\quad \text{ on  $bM_n$, $n\in\N$.}
\end{equation}
On the other hand, for each $n\in\N$ choose a compact domain $W_n$ which is a union of pairwise disjoint smoothly bounded closed discs $W_{n,p}$, $p\in X_3^{-1}(0)\cap M_n$, such that $p\in\mathring W_{n,p}\Subset M_n$. Set $W_n'=M_n\setminus \mathring W_n$, which is a connected smoothly bounded compact domain with $bW_n'=bM_n\cup bW_n=bM_n\cup \bigcup_{p\in X_3^{-1}(0)\cap M_n}bW_{n,p}$. By \ref{rcondc} and \eqref{eq:Mncup}, we may choose the discs $W_{n,p}$ so small at each step that 
\begin{equation}\label{eq:Mn''}
	 W_1' \Subset W_2' \Subset \cdots \subset  \bigcup_{n \in \n} W_n'  = M \setminus X_3^{-1} (0) =M \setminus \wt X_3^{-1} (0)
\end{equation}
and $| \wt X_3 | < 1/n$ on  $b W_n$ for all $n\in\N$.
This inequality together with \eqref{eq::normt} imply that 
$|(1, \wt X_1, \wt X_2)/ \wt X_3| >n-1$ on  $bW_n'$ for all $n\in\N$,
and hence the map $(1, \wt X_1, \wt X_2)/ \wt X_3 \colon M \setminus X_3^{-1}(0) \to \C^3$ is  almost proper in view of \eqref{eq:Mn''}. So, \ref{rcondd} holds.
  
To complete the proof it only remains to explain the induction. Note that $X^0=X|_{M_0}\colon M_0\to\C^3$ satisfies (\hyperref[seqi]{\rm b$_0$})--(\hyperref[seqap]{\rm d$_0$}), while (\hyperref[seqa]{\rm a$_0$}) and (\hyperref[seqep]{\rm e$_0$}) are void. Fix $n\ge 0$, assume that we have a number $\varepsilon_n>0$ and a holomorphic null  immersion $X^n\colon M_n\to\C^3$ satisfying \ref{seqi}--\ref{seqap}, and let us find a number $\varepsilon_{n+1}>0$ and a nonflat holomorphic null embedding $X^{n+1}\colon M_{n+1}\to\C^3$ satisfying (\hyperref[seqa]{\rm a$_{n+1}$})--(\hyperref[seqep]{\rm e$_{n+1}$}). Set 
\begin{equation}\label{eq::condkn}
		k := \max \{k(p) : p \in \Lambda \cap M_{n+1}\} 
\end{equation}
and $M_n ' = M_n \cup (\Omega \cap M_{n+1})$; recall that $\Omega\cap (bM_n\cup bM_{n+1})=\varnothing$, hence $M_n'\subset\mathring M_{n+1}$.
	Note that $M_n'\subset M$ is a (possibly disconnected) smoothly bounded Runge compact domain and $\Lambda \cap M_n ' = \Lambda\cap M_{n+1}\subset\mathring M_n'$. Extend $X^n$ to $M_n'$ by setting
\begin{equation}\label{eq:Mn'}
	X^n|_{\Omega \cap M_{n+1}\setminus M_n} = X|_{\Omega \cap M_{n+1}\setminus M_n}.
\end{equation} 
Lemma \ref{lemma::ind} applied to
$M_n'$,  $M_{n+1}$, $\Lambda\cap M_n '$, $X^n \colon M_n ' \to \C^3$, $\varepsilon_{n}/2$, and $k$
	provides a holomorphic null curve $Y = (Y_1 , Y_2, Y_3) \colon M_{n+1}\to \C^3$ such that: 
	\begin{enumerate}[label=(\Alph*)]
		\item \label{auxa} $\| Y-X^n \|_{r, M_n}<\varepsilon_{n}/2$.
		\item \label{auxi} $Y-X^n$ vanishes to order $k$ at every point of $\Lambda\cap M_{n+1}\subset M_n'$.
		\item \label{aux3} $Y_3^{-1}(0) =(X_3^n)^{-1}(0)=  X_3^{-1} (0) \cap M_n' = X_3^{-1} (0) \cap M_{n+1}\subset \Lambda\cap M_{n+1}$.

		\newcounter{Mayus1}\setcounter{Mayus1}{\value{enumi}}
	\end{enumerate}
For \ref{aux3} take into account \eqref{eq:Lambda}, \ref{seqX3}, and \eqref{eq:Mn'}. Furthermore, we may assume by the general position result in \cite[Theorem 3.4.1 (a)]{AlarconForstnericLopez2021Book} and Hurwitz theorem that 
	\begin{enumerate}[label=(\Alph*)]
		\setcounter{enumi}{\value{Mayus1}}
		\item \label{auxap} $ \mathfrak{m} (Y) > 0 $ on $b M_{n+1}$.
		\setcounter{Mayus1}{\value{enumi}}
	\end{enumerate}
By \ref{auxap} and compactness of $bM_{n+1}$ there are a number $s >0$ and a smoothly bounded compact domain $\Rcal\subset M$ such that $M_n' \Subset \Rcal \subset \mathring M_{n+1}$ and
	$\mathfrak{m} (Y) > s$  on $M_{n+1} \setminus \mathring \Rcal$.
	Fix a number $0<\delta<\varepsilon_{n}/2$ to be specified later. 
By Lemma \ref{lemma::ap} applied to
$Y\colon M_{n+1} \to \C^3$, $\Rcal\subset \mathring M_{n+1}$, $\Lambda \cap M_{n+1}\subset\mathring M_{n+1}$, $k$, $s$, $\delta$ and 
and a large enough number $\hat s>s$,  and \cite[Theorem 3.4.1 (a)]{AlarconForstnericLopez2021Book},
we obtain a nonflat holomorphic null embedding $X^{n+1} \colon M_{n+1} \to \C^3$ such that:
	\begin{enumerate}[label=(\Alph*)]
		\setcounter{enumi}{\value{Mayus1}}
		\item \label{auxxa} $\|X^{n+1} - Y \|_{r, M_n} < \delta <\varepsilon_{n}/2$.
		\item \label{indap} $\mathfrak{m}(X^{n+1})  > (n+1) \, (\| Y_3 \|_{0, b M_{n+1}}+\delta+1)$ on $b M_{n+1}$.
		\item \label{indX3} $\| X^{n+1}_3 - Y_3\|_{0, M_{n+1}} < \delta$.
		\item \label{auxxi} $X^{n+1} - Y$ vanishes to order $k$ everywhere on $\Lambda \cap M_{n+1}$.
		\setcounter{Mayus1}{\value{enumi}}
	\end{enumerate}
Choose a number $\varepsilon_{n+1}>0$ so small that (\hyperref[seqep]{\rm e$_{n+1}$}) holds.
Condition \hyperref[seqa]{\rm (a$_{n+1}$)} follows from \ref{auxa} and \ref{auxxa}. Condition \hyperref[seqa]{\rm (b$_{n+1}$)} from \eqref{eq::condkn}, \eqref{eq:Mn'}, \ref{auxi}, \ref{auxxi}, and \ref{seqi}.
	  By \ref{aux3}, to check \hyperref[seqX3]{\rm (c$_{n+1}$)}, it suffices to see that $(X_3^{n+1})^{-1}(0)= Y_3^{-1} (0)$. Indeed, \ref{aux3} and  \ref{auxxi} ensure that $Y_3^{-1} (0) \subset (X_3^{n+1})^{-1}(0) $.
	 This, \ref{auxxi},  \eqref{eq::condk}, and \eqref{eq::condkn} show that the zero divisor of $X_3^{n+1}$ is greater than or equal to that of $Y_3$ on $M_{n+1}$, whose support lies in $\Lambda\cap M_{n+1}\subset \mathring M_{n+1}$. Choosing $\delta>0$ sufficiently small, \ref{indX3} and Hurwitz theorem imply that they actually have the same divisor; in particular $(X_3^{n+1})^{-1}(0)=Y_3^{-1}(0)$ as claimed. This proves \hyperref[seqX3]{\rm (c$_{n+1}$)}.
Finally, $\| Y_3 \|_{0, b M_{n+1}} + \delta > \| X_3^{n+1}  \|_{0, b M_{n+1}} $ by \ref{indX3}, hence \ref{indap} implies \hyperref[seqap]{\rm (d$_{n+1}$)}. This closes the induction and completes the proof.
\end{proof}

%
%
\section{A Carleman theorem for complete null curves in $\C^2\times\C^*$}\label{sec:Carleman}

\noindent This section is devoted to prove 
a Carleman type theorem on better than uniform approximation for null curves in $\C^2\times\C^*$ or, more generally, with control of the zero set of the third component function. 
To state the result we need some preparations.

Let $S$ be a closed subset in an open Riemann surface $M$. The \emph{holomorphic hull
		of $S$} is the union $\wh S = \bigcup_{n \in \n} \wh S_n$, where $\{S_n\}_{n \in \n}$ is any exhaustion of $S$ by compact subsets and
	$\wh S_j$ denotes the holomorphic convex hull of $S_j$, that is,  the union of $S_j$ and its holes.
	The hull $\wh S$ is independent of the choice of exhaustion. 
	The closed set $S \subset M$ has \emph{bounded exhaustion hulls} if for every compact set 
	$K \subset M$, the set $ \wh{K \cup S} \setminus K \cup S$ is relatively compact in $M$.
A closed set $S \subset M$ is said to be \emph{Carleman admissible} if $\wh S = S$, $S$ has bounded exhaustion hulls, and $S = K \cup \Gamma$ where $K$ is the union of a locally finite pairwise disjoint collection of compact domains with piecewise smooth boundaries and $\Gamma = \overline{S \setminus K}$ is the union of a locally finite pairwise disjoint collection of smooth Jordan arcs, so that each component of $\Gamma$ intersects $bK$ only at its endpoints (if at all) and all such intersections are transverse (see \cite[Def.\ 2.1]{Castro-InfantesChenoweth2020} or \cite[\textsection3.8]{AlarconForstnericLopez2021Book}). 
A compact set $S \subset M$ is a Carleman admissible set if and only if it is a Runge admissible set in the sense of Definition \ref{def:admissible}. 
The notion of \emph{generalized null curve $S\to\C^n$ on a Carleman admissible set $S$} extends naturally that of generalized null curves on admissible sets in Definition \ref{def:gnc}, just replacing admissible by Carleman admissible.
In particular, if $(X, f \theta)$ is a generalized null curve $S\to\C^n$ on a Carleman admissible set $S\subset M$, then for any compact smoothly bounded domain $M'$ in $M$, the restriction of $(X, f \theta)$ to the admissible set $S' = S \cap M'$ is a generalized null curve $S' \to \C^n$.

The following is the main result of this section.
\begin{theorem}\label{th::Carleman} 
	Let  $M$ be an open Riemann surface, $\theta$ be a nowhere vanishing holomorphic $1$-form on $M$, $S = K \cup \Gamma$ be a Carleman admissible subset, and $\Lambda \subset M$ be a closed discrete subset.
	Also let $(X=(X_1,X_2,X_3),f \theta)$ be a generalized null curve $S \to \C^3$ of class $\Ascr^1(S)$ and a holomorphic null curve on a neighborhood of $\Lambda$, and assume that
	\begin{equation}\label{eq::cLambda}
	X_3^{-1} (0) \subset \Lambda. 
		\end{equation}
	Given a map $k\colon \Lambda \to \n$ and a continuous function $\varepsilon \colon S \to (0, +\infty)$, there is a complete holomorphic null curve  $\wt X=(\wt X_1,\wt X_2,\wt X_3) \colon M \to \C^3$ such that:
	\begin{enumerate}[label=\rm (\alph*)]
		\item  \label{cconda}$| \wt X (p)- X(p) | < \varepsilon(p)$ for all $p \in S$. 
		\item  \label{ccondi} $\wt X - X$ vanishes to order $k (p)$ at every point $p \in \Lambda$.
		\item  \label{ccondX3}$\wt X_3^{-1}(0)= X_3^{-1} (0) \subset \Lambda$. 
	\end{enumerate}
	Furthermore, if $X|_{\Lambda}  \colon \Lambda \to \C^3 $ is injective, then there is an injective holomorphic null curve $\wt X$ with the above properties.
\end{theorem}
Except for property \ref{ccondX3}, Theorem \ref{th::Carleman} is proved in \cite{Castro-InfantesChenoweth2020} (see also \cite[\textsection3.8]{AlarconForstnericLopez2021Book}),  where a Carleman type approximation result with interpolation for more general families of directed holomorphic immersions of open Riemann surfaces is shown.
	Note that the Carleman type approximation in Theorem \ref{th::Carleman} is optimal in the sense that if the original holomorphic null curve is not complete on $S$, then the approximation by complete solutions may be done in the fine $\Cscr^0 (S)$ topology, but not even in the $\Cscr^r (S)$ topology provided that $(X,f\theta)$ is of class $\Ascr^r(S)$ for some $r\ge 1$. Nevertheless, if we do not insist on the completeness condition, a modification of the proof leads to approximation in the fine $\Cscr^r(S)$ topology in that case.
\begin{proof}
We may assume that $K\neq\varnothing$, $\Lambda\subset\mathring S$, and $\varepsilon<1$ everywhere on $S$. As in the proof of Theorem \ref{th::Runge}, we assume that $X|_{\Lambda}\colon\Lambda\to\C^3$ is injective and that \eqref{eq::condk} holds.
Choose a connected component $M_0$ of $K\subset S$ and fix $p_0\in\mathring M_0$. By \cite[Lemma 3]{Chenoweth2018} (see also \cite[proof of Theorem 3.8.6]{AlarconForstnericLopez2021Book}), we can exhaust $M$ by connected, smoothly bounded, Runge compact domains $ M_1 \Subset M_2 \Subset \cdots \Subset \bigcup_{n \in \n} M_n = M$ such that $M_0 \subset\mathring  M_1$,  $(b M_n \cap S )\subset \Gamma \setminus K$, the intersection $b M_n \cap \Gamma$ is transverse and finite and the closed set  $S\cup M_n$ is Carleman admissible for all $n \in \N$.
We also make sure that the compact sets
\begin{equation}\label{eq:Sn}
	S_n:=S\cap M_{n+1}\quad\text{and}\quad S_n':=M_n\cup S_n
\end{equation}
are Runge admissible (Definition \ref{def:gnc}), and $\Lambda \cap (b S_n\cup bM_n) = \varnothing$ for all $n \ge 0$.
In particular, $M_0 =: S_{-1}  \Subset S_0 \Subset S_1 \Subset S_2 \Subset \cdots \Subset \bigcup_{n \in \n} S_n = S$ is an exhaustion of $S$.	
Set $X^0 = X|_{M_0}\colon M_0\to\C^3$. We shall inductively construct a sequence of numbers $\varepsilon_n>0$ and holomorphic null embeddings $X^n=(X_1^n,X_2^n,X_3^n) \colon M_n \to \C^3$ of class $\Ascr^1(M_n)$ satisfying the following properties for all $n \ge 1$.
\begin{enumerate}[label=(\alph*$_n$)]
		\item \label{cseqac}$|X^n (p) - X (p)|< \varepsilon (p)/2 $ for all $p \in S_{n-1} $.
		\item \label{cseqau}$|| X^n - X^{n-1} ||_{ 0, M_{n-1}} < \varepsilon_{n-1}$.
		\item \label{cseqi}	$X^n - X$ vanishes to order $k(p)$ at every point $p \in\Lambda\cap M_n\subset\mathring S_{n-1}$.
		\item \label{cseqX3}$(X^n_3)^{-1} (0) = X_3^{-1} (0)\cap M_n \subset \Lambda\cap M_n$.
       \item \label{cseqep} 
       $0<	\varepsilon_n < \varepsilon_{n-1}/2$
       and  if $Y \colon M \to \C^3$ is a holomorphic map with $ \| Y - X^n \|_{1, M_n} < 2 \varepsilon_n$, then $Y|_{ M_n}$ is an embedding and $\dist_{Y} (p_0,   b M_n) >n$.
	\end{enumerate}
Given such a sequence satisfying \ref{cseqac}, \ref{cseqau}, and \ref{cseqep} for all $n \ge 1$, the limit map \(
\wt X=\lim_{n\to\infty} X^n\colon M \to\c^3
\) is a complete holomorphic injective immersion which satisfies condition \ref{cconda} in the statement.
If the sequence also meets \ref{cseqi} and \ref{cseqX3} for all $n \ge 1$, then $\wt X$ also satisfies  \ref{ccondi} and \ref{ccondX3}, which follow  as in the proof of Theorem \ref{th::Runge} using also \eqref{eq::cLambda}.
To explain the induction,  note that $X_0 \colon M_0 \to \C^3$ meets  (\hyperref[cseqac]{\rm a$_0$}) ,(\hyperref[cseqi]{\rm c$_0$}), and (\hyperref[cseqX3]{\rm d$_0$}). 
Fix  $n \ge 0$, assume that we have  
a holomorphic null curve $X^n=(X_1^n,X_2^n,X_3^n) \colon M_n \to \C^3$ of class $\Ascr^1(M_n)$ satisfying \ref{cseqac}, \ref{cseqi}, and \ref{cseqX3}, and let us find $\varepsilon_{n+1}>0$ and a holomorphic null embedding $X^{n+1}= (X_1^{n+1}, X_2^{n+1}, X_3^{n+1})$ of class $\Ascr^1 ( M_{n+1})$ enjoying (\hyperref[cseqac]{\rm a$_{n+1}$})--(\hyperref[cseqep]{\rm e$_{n+1}$}).  
For this, note that the closure of $S_n'\setminus M_n$ intersects the smoothly bounded compact domain $M_n$ at a finite set, where some arcs of the Runge admissible set $S_n'$ are attached.
Thus, by \ref{cseqac}, we may apply the gluing lemma in \cite[Lemma 3.1]{Castro-InfantesChenoweth2020} in order to extend $X^n$ to a generalized holomorphic null curve $(X^n, f_n\theta)$ on $S_n'$ such that:
\begin{enumerate}[label=(\Alph*)]
	\item \label{gac} $| X^n(p)-X(p) | <  \varepsilon (p) /2$  for all $p\in S_n$.
	\item \label{gout} $X^n-X$ vanishes to order $k(p)$ at every point $p \in\Lambda\cap M_{n+1}\subset\mathring S_n'$.
	\newcounter{Mayus}\setcounter{Mayus}{\value{enumi}}
\end{enumerate}
The idea for this is to deform $X|_{S_n'\setminus M_n}$ very slightly in a small neighborhood of $S_n\cap bM_n$ in order to glue it with $X^n$. 
Lemma \ref{lemma::ind} applied to $S_n' \Subset M_{n+2}$, $\Lambda\cap M_{n+1}$, $X^n \colon S_n' \to \C^3$, and sufficiently large $k\in\N$ and small $\delta>0$ gives a holomorphic null curve $X^{n+1}\colon M_{n+2}\to\c^3$ satisfying (\hyperref[cseqac]{\rm a$_{n+1}$})--(\hyperref[cseqX3]{\rm d$_{n+1}$}); take into account \ref{gac} and \ref{gout}. Furthermore, by first applying \cite[Lemma 7.3.1]{AlarconForstnericLopez2021Book} to enlarge the intrinsic diameter of $X^{n+1}$, then the general position result in \cite[Theorem 3.4.1 (a)]{AlarconForstnericLopez2021Book}, and with the help of Hurwitz theorem, we can assume in addition that $X^{n+1}$ is an embedding and   $\dist_{X^{n+1}} (p_0 , bM_{n+1})>  n+1$. So, to close the induction we just choose a number $\varepsilon_{n+1}>0$ so small that (\hyperref[cseqX3]{\rm e$_{n+1}$}) holds.
This concludes the proof.
\end{proof}

%
%
	
\section{From holomorphic null curves in $\C^2\times\C^*$ to\\holomorphic null curves in $SL_2(\C)$ and CMC-1 surfaces in $\h^3$ and $\S^3_1$}\label{sec:SL2C}
	
\noindent A \emph{holomorphic null curve} in the complex special linear group
	\[
	SL_2 (\C)  = \left\{    z =  \left( \begin{array}{cc}
			z_{11} & z_{12}\\
			z_{21}& z_{22} 
		\end{array} \right) \in \C^4 :  \det z = z_{11} z_{22} - z_{12} z_{21}  = 1 \right\}
	\]
	is a holomorphic immersion $F \colon M \to SL_2 (\C)$ of an open Riemann surface $M$ 
	which is directed by the quadric variety 
	\[
	\mathcal{A} = \left\{  z =  \left( \begin{array}{cc}
			z_{11} & z_{12}\\
			z_{21}& z_{22} 
		\end{array} \right) \in \C^4 :  \det z = z_{11} z_{22} - z_{12} z_{21}  = 0   \right\};
	\]
	see Bryant \cite[Eq. (2.3)]{Bryant1987Asterisque}.
As above, to be directed by $\mathcal A$ means that the derivative $F' \colon M \to \C^4$ with respect to any local holomorphic coordinate on $M$ has range in $\mathcal A \setminus \{ 0\}$. The term {\em null} refers to the fact that tangent vectors to $F(M)$ are null (or lightlike) for the Killing metric in $SL_2 (\C)$; see \cite[\textsection2.6]{MR1834454} or \cite[\textsection 11.5]{JensenMussoNicolodi2016}. As pointed out by Mart\'in, Umehara, and Yamada in \cite[\textsection3.1]{MartinUmeharaYamada2009CVPDE}, the biholomorphism $\Tcal \colon \c^2 \times \c^* \to SL_2 ( \c) \setminus \{ z_{11} = 0\}$,
		\begin{equation}\label{eq::defT}
				\Tcal(z_1, z_2 , z_3) =   \frac{1}{z_3}\left( \begin{array}{cc}
						1& z_1 +\igot z_2\\
						z_1 - \igot z_2& z_1^2 + z_2^2 +z_3^2
					\end{array} \right),\quad (z_1,z_2,z_3)\in \C^2\times\C^*,
			\end{equation}
		takes holomorphic null curves in $\c^2 \times \c^*$ into holomorphic null curves in $SL_2 ( \c) \setminus \{ z_{11} = 0\}$. The inverse biholomorphic map $\Tcal^{-1}  \colon  SL_2 ( \c) \setminus \{ z_{11} = 0\} \to \c^2 \times \c^*$ also takes holomorphic null curves into holomorphic null curves.
		 Using this correspondence we obtain the following corollary of the proof of Theorem \ref{th::Runge}.
		%
		%
		%
		\begin{corollary}\label{co:SL2C}
			Let $M$ be an open Riemann surface, $\Lambda\subset M$ be a closed discrete set, and $F \colon U \to SL_2 (\C)  \setminus \{ z_{11} = 0\}$ be a holomorphic null curve on an open neighborhood $U \subset M$ of $\Lambda$. For any map $k \colon \Lambda \to \N$ there is a holomorphic null curve 
			\[
			\wt F = \left( \begin{array}{cc}
			\wt F_{11} & \wt F_{12}\\
			\wt F_{21}& \wt F_{22} 
		\end{array} \right)  \colon M \to SL_2 (\C) \setminus \{ z_{11} = 0\}
		\] such that:
			\begin{enumerate}[label= \rm (\roman*)]
				\item \label{condsl2i} $\wt F - F\colon U\to \C^4$ vanishes to order $k(p)$ at every point $p \in \Lambda$. 
				\item \label{condsl2ap} $(\wt F_{12}, \wt F_{21} ) \colon M \to \C^2$ is an almost proper map. In particular, the holomorphic null curve $\wt F\colon M\to SL_2(\C)\subset\C^4$ is almost proper, and hence complete.
			\end{enumerate}
			Furthermore, if $F|_{\Lambda}$ is injective then $\wt F$ can be chosen to be injective.
		\end{corollary}
		\begin{proof}
			We assume that $F|_{\Lambda}$ is injective, otherwise the proof is simpler. Set $X=(X_1,X_2,X_3):= \Tcal^{-1} \circ F    \colon U \to \C^2 \times \C^*$, where $\Tcal$ is given in \eqref{eq::defT}. Note that $X|_{\Lambda} \colon \Lambda \to \C^3$  is injective. 
Choose an exhaustion $M_1\Subset M_2\Subset\cdots $ of $M$ as in \eqref{eq:Mncup}.			
			 Since $X_3^{-1}(0)=\varnothing$, an inspection of the proof of Theorem \ref{th::Runge} for the data
			$
			S= \varnothing$, $\Lambda \subset M$, $X \colon U \to \C^3$, and $k \colon \Lambda  \to \N$
			gives an injective holomorphic null curve $\wt X = (\wt X_1, \wt X_2 , \wt X_3) \colon M \to \C^2\times\C^*$ such that: 
			\begin{enumerate}[label= \rm (\alph*)]
				\item \label{sl2int}$ \wt X - X$ vanishes to order $k(p)$ at every point $p \in \Lambda$. 
				\item \label{sl2ap}$| (\wt X_1, \wt X_2)/ \wt X_3| >n-1$ on $bM_n$ for all $n\in\N$; see \eqref{eq::normt}.
			\end{enumerate}
			 The injective holomorphic null curve $\wt F := \Tcal \circ \wt X \colon M \to SL_2 ( \C) \setminus \{ z_{11}= 0\}$ then satisfies the conclusion of the corollary.  Indeed, condition \ref{condsl2i} follows from the biholomorphicity of $\Tcal$ and \ref{sl2int}.
			  On the other hand, by \eqref{eq::defT} we have that
			 $
	 | \wt F_{12} |^2 + | \wt F_{21} |^2 =  2 ( | \wt X_1|^2 + | \wt X_2 |^2)/| \wt X_3|^2$. This, \ref{sl2ap}, and \eqref{eq:Mncup}, imply \ref{condsl2ap}.
		\end{proof}

Denote by $\mathbb{L}^4$ the Minkowski space of dimension $4$ with the canonical Lorentzian metric of signature $(-+++)$. 
 We identify $\mathbb{L}^4$ with the space of hermitian matrices ${\rm Her}(2)\subset M_2(\C)$ by
\[
 \mathbb{L}^4 \ni (x_0, x_1, x_2, x_3) \longleftrightarrow  \left(  \begin{array}{cc}
 	x_0 + x_3 & x_1 + \igot x_2 \\
 	x_1 - \igot x_2 & x_0- x_3
 \end{array} \right) \in {\rm Her}(2),
\]
and consider the hyperboloid model for the hyperbolic $3$-space
		\[
		\h^3 = \left\{ (x_0, x_1, x_2, x_3) \in 	\mathbb{L}^4 : - x_0^2 +  x_1^2 + x_2^2 + x_3^2 = -1, x_0  > 0  \right\}
		\]
		with metric induced by $\mathbb{L}^4$.
		Under the above identification we have that
		 \[
			\h^3 = \{ A \overline{A}^T : A \in SL_2 ( \C) \} = SL_2( \C) / SU(2),
		 \]
		 where $\overline \cdot$ and $\cdot^T$ mean complex conjugation and transpose matrix respectively.
		In this setting, the canonical projection
		\begin{equation}\label{eq::bproj}
			\pi_H \colon  SL_2 (\C) \to \h^3,\quad \pi_H (A) =  A \overline{A}^T,
		\end{equation}
			takes holomorphic null curves in $SL_2 (\C)$ to conformal $\CMC$ immersions in $\h^3$ (i.e., Bryant surfaces); see \cite{Bryant1987Asterisque}. 
			Since $SU(2)$ is compact, $\pi_H$ is a proper map, so it takes almost proper holomorphic null curves in $SL_2 (\C)$ into almost proper $\CMC$ immersions.
			Conversely, every simply connected Bryant surface in $\h^3$ lifts to a holomorphic null curve in $SL_2(\C)$.
			Finally, note that the restricted map $\pi_H |_{ SL_2 (\C) \setminus \{ z_{11} = 0\}} \colon SL_2 (\C) \setminus \{ z_{11} = 0\} \to \h^3$ is surjective. These observations lead to the following more precise version of Corollary \ref{co:intro-dense}.
%
%
\begin{corollary}\label{cor::intbryant}
Let $M$ be an open Riemann surface, $\Lambda\subset M$ be a closed discrete subset, and $\varphi \colon U \to \h^3$ be a conformal $\CMC$ immersion on an open neighborhood $U \subset M$ of $\Lambda$. For any map $k \colon \Lambda \to \N$ there is an almost proper (hence complete) conformal $\CMC$ immersion  $\wt \varphi \colon M \to \h^3$ agreeing with $\varphi$ to order $k(p)$  at every point $p \in \Lambda$. 
In particular, there is an almost proper (hence complete) conformal  $\CMC$ immersion $M\to\h^3$ with everywhere dense image.
\end{corollary}
Corollary \ref{cor::intbryant} proves Theorem \ref{th:intro-Bryant} in the special case when $E=\varnothing$. The general case of the theorem requires some extra work and we prove it in Section \ref{sec:bryant}.
\begin{proof}
We may assume without loss of generality that $U$ is simply connected, otherwise we restrict $\varphi$ to a simply connected domain. Since  $\pi_H |_{ SL_2 (\C) \setminus \{ z_{11} = 0\}}$ is surjective, there exists a holomorphic null curve $F \colon U \to SL_2 (\C)$ such that $\pi_H \circ F = \varphi$ and $F(\Lambda)\subset SL_2(\C)\setminus \{ z_{11} = 0\}$. Up to replacing $U$ by a smaller open neighborhood of $\Lambda$ if necessary, we may assume that $F$ has range in $SL_2(\C)\setminus \{ z_{11} = 0\}$. Corollary \ref{co:SL2C}
then furnishes us with an almost proper holomorphic null curve $\wt F \colon M \to SL_2 ( \C)\subset\C^4$ (with range in $SL_2 ( \C)\setminus \{ z_{11} = 0\}$) such that $\wt F - F$ vanishes to order $k(p)$ at every point $p \in \Lambda$. The conformal $\CMC$ immersion $\wt \varphi := \pi_H  \circ \wt F \colon M \to \h^3$ clearly satisfies the conclusion of the corollary.
		\end{proof}
	Now consider the de Sitter $3$-space
	\[
	\S_1^3 = \{   ( x_0, x_1, x_2, x_3) \in \mathbb{L}^4 : -x_0^2 + x_1^2 + x_2^2 + x_3^2 = 1 \}
	\] 
	with metric induced from $\mathbb{L}^4$ and identify $\S_1^3=SL_2(\C)/SU_{1,1}$ as in \cite[\textsection 0]{FRUYY09}. By \cite[Proposition 1.7]{Fujimori06} (see also \cite[Theorem 1.9]{Fujimori06}), the  canonical projection 
\[
		\pi_S  \colon SL_2 ( \C) \to \S_1^3, \quad \pi_S (A) = A \left(\begin{array}{cc}
	1&0\\
	0&-1
\end{array} \right)   \overline{A}^T
\]
takes holomorphic null curves in $SL_2 ( \C)$ into \emph{$\CMC$ faces} in $\S_1^3$ in the sense of Fujimori \cite[Def.\ 1.4]{Fujimori06}; see also Fujimori et al. \cite[Def.\ 1.1]{FRUYY09}. Moreover, as shown by Yu in \cite[p.\ 2999]{Yu97}, complete holomorphic null curves in $SL_2(\C)$ project to weakly complete $\CMC$ faces in $\S_1^3$ in the sense of \cite[Def.\ 1.3]{FRUYY09}. Conversely, every simply connected $\CMC$ face in $\S_1^3$ lifts to a holomorphic null curve in $SL_2 ( \C)$ \cite[Theorem 1.9]{Fujimori06}. Finally, $\pi_S |_{SL_2 ( \C) \setminus \{ z_{11} = 0\}} \colon SL_2 ( \C) \setminus \{ z_{11} = 0\} \to \S_1^3$ is easily seen to be surjective. Therefore, reasoning as in the proof of Corollary  \ref{cor::intbryant} we can establish the following analogous result for $\CMC$ faces in $\S_1^3$.

%
%
\begin{corollary}
Let $M$ be an open Riemann surface, $\Lambda\subset M$ be a closed discrete subset, and $f \colon U \to \S_1^3$ be a  $\CMC$ face on an open neighborhood $U \subset M$ of $\Lambda$. For any map $k \colon \Lambda \to \N$ there exists a weakly complete  $\CMC$ face  $\wt f \colon M \to \S_1^3$ agreeing with $f$ to order $k(p)$ at every point $p \in \Lambda$. 
In particular, there is a weakly complete $\CMC$ face  $M \to \S_1^3$ with everywhere dense image.
\end{corollary}

%
%
\begin{remark}
A point of view which unifies $\CMC$ surfaces in $\h^3$ and $\CMC$ faces in $\S_1^3$ is that of holomorphic null (also called isotropic) curves in the nonsingular complex hyperquadric $\mathbb{Q}_3\subset\CP^4$, regarded as the Grassmannian of Lagrangian 2-planes of $\C^4$, with its standard holomorphic conformal structure. We refer to the paper by  Musso and Nicolodi \cite{MussoNicolodi2022IJM}, where the relations between holomorphic null curves in $\mathbb Q_3$ and a number of relevant classes of surfaces in Riemannian and Lorentzian spaceforms are discussed.
\end{remark}

%
%
\section{Completion of the proof of Theorem \ref{th:intro-Bryant}}\label{sec:bryant}
\noindent 
We begin with some preparations. Denote $\D=\{z\in\C\colon |z|<1\}$ and $\D^*=\D\setminus\{0\}$. Let $F=(F_{ij})\colon \D^*\to SL_2(\C) \subset \C^4$ be a holomorphic null curve extending meromorphically to $\D$ with an effective pole at the origin. In particular, the end of $F$ at the origin is proper, hence complete. According to the discussion preceding Corollary \ref{cor::intbryant} (in particular, see \eqref{eq::bproj}), we have that 
\begin{equation}\label{eq:varphipiF}
	\varphi:=\pi_H\circ F\colon \D^*\to \h^3
\end{equation}
is a conformal $\CMC$ immersion which is proper (hence complete) at the origin as well. Assume that $\varphi$ is not totally umbilic and let us look more carefully at the geometry of $\varphi$ near the origin, by following the developments in \cite{UmeharaYamada1993AM}.
We can write	
\[
				F^{-1}  d F = \left( \begin{array}{cc}
					g &  -g^2\\
					1& -g
				\end{array}\right) \omega ,
\]
where 
 \begin{equation}\label{eq::g}
	g= -\frac{dF_{12}}{dF_{11}} = - \frac{d F_{22}}{dF_{21}}\quad \text{is a meromorphic function on $\D$}
\end{equation}
(the so-called {\em secondary Gauss map} of $\varphi$; see \cite[Def.\ 1.2]{UmeharaYamada1996Ann}) and
\begin{equation}\label{eq::w}
			 	\omega = F_{11} d F_{21} - F_{21} d F_{11}\quad \text{is a meromorphic $1$-form on $\D$}
\end{equation}
which is holomorphic on $\D^*$. Up to replacing $\D$ by $a\D^*$ for some small $a>0$ if necessary we may assume that $g$ is holomorphic on $\D^*$. The {\em hyperbolic Gauss map} $G\colon\D^*\to \partial_\infty\h^3$ is then given by
\[
	G=\frac{dF_{11}}{dF_{21}}=\frac{dF_{12}}{dF_{22}},
\]
see \cite{Bryant1987Asterisque} or \cite[Eq.\ (1.12)]{UmeharaYamada1993AM},
and hence it extends to a meromorphic function on $\D$.  This means that $\varphi$ has a regular end at the origin \cite[Def.\ 1.4]{UmeharaYamada1993AM}.
Moreover, the Riemannian metric induced on $\D^*$ by the hyperbolic metric in $\h^3$ via $\varphi$ is
\begin{equation}\label{eq::metric}
	(1 + |g|^2 )^2 |\omega |^2, 
\end{equation} 
see \cite[Eq.\ (1.9)]{UmeharaYamada1993AM}. By \eqref{eq::g}, \eqref{eq::w}, and \eqref{eq::metric}, the end of $\varphi$ at the origin is of finite total curvature.
The
 \emph{Hopf differential} of $\varphi$ (see \cite{Bryant1987Asterisque} 
 or \cite[Eq.\ (1.10) and (2.2)]{UmeharaYamada1993AM}) is given by the holomorphic $2$-form
\begin{equation}\label{eq::hdif}
	Q = \omega dg =  \Big(    \sum_{j= -2}^{\infty}  \hat q_j z^j  \Big) dz^2\quad \text{on $\D^*$},
\end{equation}
where $\hat q_j \in \C$ for every integer $j\ge -2$.
By \eqref{eq::g}, \eqref{eq::w}, and \eqref{eq::metric}, and since $\varphi$ is complete at the origin, there are integers $\mu,\nu$ with $\min\{\nu,2\mu+\nu\}\le -1$ such that
\begin{equation}\label{w1}
g (z) = z^{\mu} \hat g (z)\quad\text{and}\quad
	\omega (z) = z^{\nu} \hat w (z) \, dz,
\end{equation}
where $\hat g, \hat w \colon \D \to \c$ are holomorphic functions which do not vanish at the origin; see \cite[Eq.\ (W.1)--(W.3)]{UmeharaYamada1993AM}.  
The {\em indicial equations} of $(g,\omega)$ are
\begin{equation}\label{eq::indeq}
	t^2 - ( \nu +1 ) t - \hat q_{-2} = 0 , \quad
	t^2 - ( 2 \mu + \nu + 1) t - \hat q_{-2} = 0;
\end{equation}
see  \cite[Eq.\ (e.1) and (e.2)]{UmeharaYamada1993AM}.
Denote their solutions by $\lambda_1, \lambda_1 - m_1 \in \C$ and $\lambda_2, \lambda_2 - m_2 \in \C$, respectively. In this situation, \cite[Theorem 2.4]{UmeharaYamada1993AM} ensures that $m_1, m_2$ lie in $\Z\setminus\{0\}$, and hence we may assume that both $m_1$ and $m_2$ are positive integers. (These numbers are nothing but the positive square roots of the discriminants of the equations in \eqref{eq::indeq}.)
The positive integer
\begin{equation}\label{eq:m}
m := \min \{ m_1, m_2\} \in \N
\end{equation}
is the {\em multiplicity} of the end of $\varphi$ at the origin; see \cite[Def.\ 5.1]{UmeharaYamada1993AM}. By virtue of \cite[Theorem 5.2]{UmeharaYamada1993AM}, the end of $\varphi$ at the origin is embedded if and only if $m= 1$.

From now on, we assume in addition that the given holomorphic null curve $F:\D^*\to SL_2(\C)$ is of the form 
\begin{equation}\label{eq:F=ToX|}
	F=\Tcal\circ X|_{\D^*}\colon\D^*\to SL_2(\C),
\end{equation}
where $\Tcal$ is the biholomorphism in \eqref{eq::defT} and $X = (X_1, X_2, X_3) \colon \D \to \C^3$ is a holomorphic null curve with $X_3^{-1}(0)=  \{ 0 \} \subset  \D$.  
In particular, $F$ has range in $SL_2(\C)\setminus\{z_{11}=0\}$. Since $X$ is a null immersion we have that 
\begin{equation}\label{eq:nula}
(X_1' - \igot X_2') (X_1 ' + \igot X_2 ') = - (X_3 ' )^2
\quad\text{and}\quad \sum_{j=1}^3|X_j'|\neq 0. 
\end{equation}
From the former condition above, \eqref{eq::g}, \eqref{eq::w}, and \eqref{eq::hdif}, it follows that
\begin{equation}\label{eq::Bdata}
	g  =  \frac{-X_3 ' X_3 }{X_1 ' - \igot X_2 '} - (X_1 + \igot  X_2), \quad
	\omega = \frac{ X_1 ' - \igot X_2 '}{ X_3^{2}} dz,
\end{equation}
and
\begin{equation}\label{eq::hopf}
	Q = \left( - \frac{X_3 ''}{X_3 }  + \frac{X_3 '}{X_3 } \cdot \frac{X_1 '' - \igot  X_2 ''}{X_1 ' - \igot X_2 '}  \right) dz^2.
\end{equation}
Recall the notation in \eqref{eq:orden} and set
\begin{equation}\label{eq:kl}
	c = \ord_0 (X_3) \ge 1 \quad\text{and} \quad	l = \ord_0 (X_1' - \igot X_2 ') \ge 0.
\end{equation}
By \eqref{eq:nula}, we have that $X_1' - \igot X_2'$ and $X_1 ' + \igot X_2'$ have no common zeros on $\D$,   
\begin{equation}\label{eq:l=02k-2}
	l\in\{0,2c-2\}, \quad \text{and}\quad \ord_0(X_1 ' + \igot X_2 ')=2c-2-l\in\{0,2c-2\}.
\end{equation}
On the other hand, \eqref{eq::hopf} and \eqref{eq:kl} give that $\hat{q}_{-2} = c(1-c +l)$. Thus, by \eqref{eq:l=02k-2} it turns out that
\begin{equation}\label{eq:q2=0}
	\hat{q}_{-2}=0 \text{ if and only if } c=1 \text{ (and hence $l=0$)}.
\end{equation} 
By \eqref{w1}, \eqref{eq::Bdata}, \eqref{eq:kl}, and \eqref{eq:l=02k-2} we have that
$\nu = l-2c \le -2$ and either $\mu = 0$ or $\mu \ge 2c -1-l  = -(\nu +1 ) \ge 1$.
In any case, $(2 \mu + \nu +1)^2 \geq (\nu+1)^2$, so \eqref{eq::indeq}, \eqref{eq:m}, and the fact that $\hat{q}_{-2} = c(1-c +l)$ ensure that 
$
	m = m_1 = \sqrt{(\nu +1)^2 + 4 \hat q_{-2}}=l+1.
$ 
This, \eqref{eq:kl}, and \eqref{eq:l=02k-2} show the following.
%
%
\begin{lemma}\label{lemma::m}
Let $\varphi  \colon \D^* \to \h^3$ be a conformal $\CMC$ immersion of the form \eqref{eq:varphipiF} with $F$ of the form \eqref{eq:F=ToX|}. If $\varphi$ is not totally umbilic, then the secondary Gauss map of $\varphi$ is holomorphic at the origin and 
the end of $\varphi$ at the origin is complete, of finite total curvature, regular, and of multiplicity $\ord_0 (X_1 ' - \igot X_2') +1$, which is odd. 
\end{lemma}
Since the holomorphic null curve $F=\Tcal\circ X|_{\D^*} \colon \d^* \to SL_2 ( \c)$ extends meromorphically to $\d$ with an effective pole at the origin,  \cite[Lemma 3.2 (i)]{BohlePeters2009} ensures that the end of the conformal $\CMC$ immersion $\varphi=\pi_H\circ F\colon \D^*\to \h^3$ at the origin is smooth (see Section \ref{sec:intro} or the definition in  \cite[p.\ 590]{BohlePeters2009}) if and only if $F' F^{-1} \colon \D \to \C^4$ has a pole of order 2 at the origin; i.e., $\min \{ \ord_0 (S_{ij}): i,j =1,2\} = -2$ where $F' F^{-1}=(S_{ij})$. A straightforward computation shows that
\[
			S_{11}= - S_{22} =\frac{-X_3 '} {X_3} - \frac{(X_1-\igot X_2)(X_1 ' + \igot X_2' )}{X_3^2},\qquad S_{12} = \frac{X_1 ' + \igot  X_2 '}{X_3^2},
\]
\[
S_{21} = (X_1 ' -\igot  X_2 ')-\frac{(X_1-\igot  X_2)^2(X_1 ' +\igot  X_2 ')}{X_3^2}-2\frac{(X_1-\igot  X_2)X_3'}{X_3}.
\]
It then follows from \eqref{eq:kl} and \eqref{eq:l=02k-2} that $\min\{\ord_0(S_{ij}): i,j=1,2\}=\ord_0(S_{12})=-(l+2)$. Thus, the end of $\varphi$ at the origin is smooth if and only if $l=0$. This, \eqref{eq:kl}, Lemma \ref{lemma::m}, and the aforementioned \cite[Theorem 5.2]{UmeharaYamada1993AM} prove the following. 
%
%
\begin{lemma}\label{lemma::smooth}
Let $\varphi  \colon \D^* \to \h^3$ be a conformal $\CMC$ immersion of the form \eqref{eq:varphipiF} with $F$ of the form \eqref{eq:F=ToX|}. If $\varphi$ is not totally umbilic, then the end of $\varphi$ at the origin is smooth if and only if it is of multiplicity $1$.
\end{lemma}	

%
%
\begin{proof}[Proof of Theorem \ref{th:intro-Bryant}]
Assume that $E\neq\varnothing$; otherwise the result follows from Corollary \ref{cor::intbryant}. Arguing as in the proof of Corollaries \ref{co:SL2C} and \ref{cor::intbryant}, up to passing to a smaller neighborhood $U$ of $\Lambda$ in $M$ if necessary, we may assume that $E\cap\overline U=\varnothing$ and $\varphi = \pi_H  \circ \Tcal \circ  X\colon U\to\h^3$ for a holomorphic null curve $X  \colon U \to \c^2\times\C^*$; see \eqref{eq::defT} and \eqref{eq::bproj}.
 Let $V=\bigcup_{p\in E} V_p$ be a domain which is a union of smoothly bounded holomorphic coordinate discs $V_p\ni p$, $p\in E$, with pairwise disjoint closures such that $\overline V\cap\overline U=\varnothing$, and let $z\colon V_p\to\D$ be a holomorphic coordinate on $V_p$ with $z(p)=0$, $p\in E$. We extend $X\colon U\to \C^2\times\C^*$ to $V$ as follows. For $p\in E$, we define
\begin{equation}\label{eq::xm1}
	X(z) =  \Big(  \frac{z}{2} - 2 \frac{z^3}{3}  \,,\,   \igot \Big(\frac{z}{2}  + 2 \frac{z^3}{3}  \Big)   \,,\,  z^2 \Big),\quad z\in V_p,
\end{equation}
if $m(p)=1$, 
while if $m(p)>1$ then we define
\begin{equation}\label{eq::xm}
	X(z) =  \Big(  \frac{z^{m(p)}}{m(p)} -  \alpha(p)^2 z \,,\,   \igot \Big( \frac{z^{m(p)}}{m(p)}  + \alpha(p)^2 z \Big)  \,,\, 2 z^{\alpha(p)} \Big),\quad z\in V_p,
\end{equation}
where $\alpha(p) = (m(p) + 1)/2 \in \n$ (recall that $m(p)$ is odd for all $p\in E$).
This gives a holomorphic null curve $X=(X_1,X_2,X_3)\colon U\cup V\to\C^3$ such that  $X_3^{-1} (0) = E$ and
\begin{equation}\label{eq::condm} 
	\ord_p (X_1' -\igot X_2') = m(p)-1\quad \text{for all $p\in E$}.
\end{equation}
Extend the given map $k\colon \Lambda\to\N$ to $E$ by setting $k(p)=m(p)+1\ge 2$ for all $p\in E$. 
Theorem \ref{th::Runge} applied to $X \colon U \cup V\to \C^3$ and $k \colon \Lambda \cup E\to \n$
furnishes us with a holomorphic null curve $\wt X = (\wt X_1, \wt X_2, \wt X_3) \colon M  \to  \C^3$ satisfying the following:
\begin{enumerate}[label=(\Roman*)]
		\item \label{contact} $\wt X - X$ vanishes to order $k (p)$ at every point $p \in \Lambda \cup E$.
		\item\label{zeros}  $\wt X_3^{-1} (0) = X_3^{-1} (0) = E$.
		\item \label{ap} $( 1, \wt X_1, \wt X_2 )/\wt X_3\colon M\setminus E\to \C^3$ is an almost proper map.
\end{enumerate}
We claim that the conformal $\CMC$ immersion  
	\[
	\psi := \pi_H \circ \Tcal \circ \wt X |_{M \setminus E} \colon  M \setminus E \to \h^3
	\]
satisfies the conclusion of the theorem; note that it is well defined by \ref{zeros}. Indeed, condition \ref{prop1} in the statement of the theorem is implied by \ref{contact}. By \eqref{eq::condm}, \ref{contact}, and \ref{zeros} we have that, locally around each point $p\in E$,  $\psi$ is of the form \eqref{eq:varphipiF} with $F$ of the form \eqref{eq:F=ToX|}, 
and $\ord_p (\wt X_1' -\igot \wt X_2') = m(p)-1$. Further, by \eqref{eq::hdif}, \eqref{eq:kl}, \eqref{eq:q2=0}, \eqref{eq::xm1}, \eqref{eq::xm}, \ref{contact}, and the fact that $\alpha(p)\ge 2$ for all $p\in E\setminus m^{-1}(1)$, we have that the Hopf differential of $\psi$ is not identically zero; equivalently, $\psi$ is not totally umbilic.  Lemmas \ref{lemma::m} and \ref{lemma::smooth} then imply conditions \ref{prop2} and \ref{prop3}. Finally, \ref{ap} ensures that the holomorphic null curve $ \Tcal \circ \wt X |_{M \setminus E} \colon  M \setminus E \to SL_2(\C)$ is almost proper, and hence $\psi\colon M\setminus E \to\h^3$ is almost proper as well by properness of $\pi_H$.
 \end{proof}

\begin{remark}
The proof we have given provides a conformal $\CMC$ immersion $\psi\colon M\setminus E\to\h^3$ satisfying the conclusion of Theorem \ref{th:intro-Bryant} which, in addition, lifts to a holomorphic null curve $M\setminus E\to SL_2(\C)$ with range in $SL_2(\C)\setminus\{z_{11}=0\}$.
\end{remark}


\subsection*{Acknowledgements}
Alarc\'on and Hidalgo are partially supported by the State Research Agency (AEI) via the grant no.\ PID2020-117868GB-I00 and the ``Maria de Maeztu'' Excellence Unit IMAG, reference CEX2020-001105-M, funded by MCIN/AEI/10.13039/501100011033/, Spain. 

Castro-Infantes is partially supported by the grant PID2021-124157NB-I00 funded by MCIN/AEI/10.13039/501100011033/ ‘ERDF A way of making Europe’, Spain; and by Comunidad Aut\'{o}noma de la Regi\'{o}n de Murcia, Spain, within the framework of the Regional Programme in Promotion of the Scientific and Technical Research (Action Plan 2022), by Fundaci\'{o}n S\'{e}neca, Regional Agency of Science and Technology, REF, 21899/PI/22.

We thank the anonymous referee for directing our attention to the paper \cite{MussoNicolodi2022IJM}. 



\noindent Antonio Alarc\'{o}n

\noindent Departamento de Geometr\'{\i}a y Topolog\'{\i}a e Instituto de Matem\'aticas (IMAG), Universidad de Granada, Campus de Fuentenueva s/n, E--18071 Granada, Spain.

\noindent  e-mail: {\tt alarcon@ugr.es}

\bigskip
\noindent Ildefonso Castro-Infantes

\noindent Departamento de Matem\'aticas, Universidad de Murcia, C. Campus Universitario, 9, 30100 Murcia, Spain.

\noindent  e-mail: {\tt ildefonso.castro@um.es }

\bigskip
\noindent Jorge Hidalgo

\noindent Departamento de Geometr\'{\i}a y Topolog\'{\i}a e Instituto de Matem\'aticas (IMAG), Universidad de Granada, Campus de Fuentenueva s/n, E--18071 Granada, Spain.

\noindent  e-mail: {\tt jorgehcal@ugr.es}

\end{document}